\documentclass[12pt]{amsart}
\usepackage{amsmath,amsfonts,amssymb,amsthm,mathtools,a4wide,url,mathscinet,mathdots,tensor,MnSymbol,wasysym,marvosym}
\usepackage[all]{xy}
\usepackage{setspace}
\usepackage[usenames]{color}

\makeatletter
\newenvironment{subtheorem}[1]{%
  \def\subtheoremcounter{#1}%
  \refstepcounter{#1}%
  \protected@edef\theparentnumber{\csname the#1\endcsname}%
  \setcounter{parentnumber}{\value{#1}}%
  \setcounter{#1}{0}%
  \expandafter\def\csname the#1\endcsname{\theparentnumber.\Alph{#1}}%
  \ignorespaces
}{%
  \setcounter{\subtheoremcounter}{\value{parentnumber}}%
  \ignorespacesafterend
}
\makeatother
\newcounter{parentnumber}

\newtheorem{theorem}{Theorem}[section]
\newtheorem{lemma}[theorem]{Lemma}
\newtheorem{definition}[theorem]{Definition}
\newtheorem{proposition}[theorem]{Proposition}
\newtheorem{corollary}[theorem]{Corollary}
\theoremstyle{remark}
\newtheorem{remark}[theorem]{Remark}
\newtheorem{example}[theorem]{Example}

\newcommand{\C}{\mathbb{C}}

\newcommand{\Z}{\mathbb{Z}}
\newcommand{\Fl}{\mathcal{F}}

\newcommand{\ess}{\mathcal{E}}
\newcommand{\smsqro}{\square}
\newcommand{\bsqro}{\setlength\fboxrule{1.1pt}\setlength\fboxsep{0pt}\fbox{\phantom{\rule{5pt}{5pt}}}}
\newcommand{\smsqr}{\smsqro_n}
\newcommand{\bsqr}{\bsqro_n}
\newcommand{\trnc}[1]{\overline{#1}}
\newcommand{\rk}{r}
\newcommand{\pair}{\Sigma}
\newcommand{\grph}{\Gamma}
\newcommand{\RR}{\mathcal{R}}
\newcommand{\Trans}{\RR_n}
\newcommand{\cvr}{\triangleleft}
\newcommand{\RRs}{\RR^\cvr}
\newcommand{\GL}{\operatorname{GL}}
\newcommand{\pairsets}{\mathfrak{P}}
\newcommand{\pairset}{\pairsets_n}
\newcommand{\lshft}[1]{\overset{\leftarrow}{#1}}               
\newcommand{\lushft}[1]{{#1}^{\nwarrow}}
\newcommand{\rshft}[1]{\overset{\rightarrow}{#1}}              
\newcommand{\rdshft}[1]{{#1}^{\searrow}}
\newcommand{\prp}{\mathcal{P}}

\newcommand{\prpu}{\mathcal{P}^{\uparrow}}
\newcommand{\prpd}{\mathcal{P}^{\downarrow}}
\newcommand{\qrp}{\mathcal{Q}}
\newcommand{\qrpu}{\mathcal{Q}^{\uparrow}}
\newcommand{\qrpd}{\mathcal{Q}^{\downarrow}}
\newcommand{\auxs}{\mathcal{A}}
\newcommand{\auxss}{\auxs^\cvr}
\newcommand{\nextd}[1]{{#1}^{\downarrow}}
\newcommand{\nextrd}[1]{{#1}^{\lcurvearrowse}}
\newcommand{\nextu}[1]{{#1}^{\uparrow}}
\newcommand{\nextlu}[1]{{#1}^{\lcurvearrownw}}
\newcommand{\rest}{\big|}
\newcommand{\pnt}{\underline{p}}
\newcommand{\X}{L}
\newcommand{\Xc}{\tilde L}
\newcommand{\diff}{\Delta}
\newcommand{\crnrs}{\mathcal{C}}
\newcommand{\crit}{\mathfrak{C}}
\newcommand{\red}{r}
\newcommand{\cmpr}{\perp}
\newcommand{\infl}{\mathcal{I}}

\makeatletter
\renewcommand*\env@matrix[1][\arraystretch]{%
  \edef\arraystretch{#1}%
  \hskip -\arraycolsep
  \let\@ifnextchar\new@ifnextchar
  \array{*\c@MaxMatrixCols c}}
\makeatother
\begin{document}

\author{Erez Lapid}
\title{A tightness property of relatively smooth permutations}

\begin{abstract}
It is well known that many geometric properties of Schubert varieties of type $A$ can be interpreted combinatorially.
Given two permutations $w,x\in S_n$ we give a combinatorial consequence of the property that the smooth locus of the
Schubert variety $X_w$ contains the Schubert cell $Y_x$. This is a necessary ingredient for the interpretation of
recent representation-theoretic results of the author with M\'inguez in terms of identities of Kazhdan--Lusztig polynomials.
\end{abstract}

\maketitle

\setcounter{tocdepth}{1}
\tableofcontents

\section{Introduction}

Consider the flag variety $\Fl_n$ (over $\C$) consisting of all flags
\[
0=V_0\subset V_1\subset\dots\subset V_n=\C^n,\ \ \dim V_i=i
\]
of an $n$-dimensional vector space. It can be identified with $\GL_n(\C)/B_n$ where $B_n\subset\GL_n(\C)$ is the subgroup of upper triangular matrices,
which is the stabilizer of the standard flag $V_i=\C^i$.
The orbits of $B_n$ on $\Fl_n$ are parameterized by permutation matrices. Let $Y_w$ and $X_w$, respectively, be the Schubert cell (the orbit)
and the Schubert variety (its closure) corresponding to a permutation $w\in S_n$.
The dimension of $Y_w$ (and $X_w$) is the length of $w$.
The inclusion relation $Y_x\subset X_w$ gives rise to the Bruhat 
partial order on $S_n$ which can be described combinatorially as
$\rk_w\le\rk_x$ where $\rk_w:\{0,\dots,n\}\times\{0,\dots,n\}\rightarrow\Z_{\ge0}$ is the rank function
\[
\rk_w(i,j)=\#\{u=1,\dots,i:w(u)\le j\}\le\min(i,j)=\rk_e(i,j).
\]
Thus, $X_w$ consists of the flags satisfying the conditions
\begin{equation} \label{eq: condXw}
\dim (V_i\cap \C^j)\ge\rk_w(i,j),\ i,j=1,\dots,n.
\end{equation}
In practice, many of these conditions are redundant \cite{MR1154177}.

It is well-known that smooth Schubert varieties are defined by inclusions, i.e., by relations of the form
$V_i\subset\C^j$ or $\C^j\subset V_i$ \cite{MR870962, MR1013667} -- see also \cite{MR1934291}.
In other words, in the conditions \eqref{eq: condXw} above it suffices to take $i,j$ for which $\rk_w(i,j)=\rk_e(i,j)$.
This fact admits the following generalization which is our main result.
\begin{theorem} \label{thm: mainintro}
Suppose that $Y_x$ is contained in the smooth locus of $X_w$.\footnote{Recall that $X_w$ is smooth if and only if
it is smooth at the point $Y_e$.}
Let $y\in S_n$ and assume that $\rk_y(i,j)=\rk_w(i,j)$ whenever $\rk_w(i,j)=\rk_x(i,j)$.
Then $Y_y\subset X_w$.
\end{theorem}

Of course, unlike in the case $x=e$, the condition on $y$ is in general not necessary
for the inclusion $Y_y\subset X_w$.

The condition on the pair $(w,x)$ can be spelled out combinatorially using the well-known description of the tangent space of
$X_w$ at $Y_x$ \cite{MR752799}.

Theorem \ref{thm: mainintro} is a necessary combinatorial ingredient for the translation of the representation-theoretic result
of \cite{1605.08545} to a certain identity of Kazhdan--Lusztig polynomials with respect to the symmetric group.
We refer the reader to [ibid., \S10] for a self-contained statement of the identity and for an explanation of the role of Theorem
\ref{thm: mainintro}.\footnote{In fact, for this application one may assume in addition that $x\le y$
and that $x$ is $213$-avoiding.}
At any rate, we hope that Theorem \ref{thm: mainintro} is interesting in its own right.
Our proof is purely combinatorial relying on a simple construction due to Gasharov \cite{MR1827861},
which was used in the solution of the Lakshimbai--Sandhya conjectural description of the singular locus of $X_w$,
completed independently by Billey--Warrington \cite{MR1990570}, Manivel \cite{MR1853139}, Kassel--Lascoux--Reutenauer \cite{MR2015302} and Cortez \cite{MR1994224}.
(However, we are not aware of a deeper connection between these results and ours.)
We will give an outline of the proof in the next section after introducing some notation.
It would be desirable to give a geometric context of the statement, if not the proof, of Theorem \ref{thm: mainintro}.

\section{Statement of main result}

\subsection{Notation and preliminaries} Let $S_n$ be the symmetric group on $\{1,\dots,n\}$ with length function $\ell$.
We denote the Bruhat order on $S_n$ by $\le$ and write $y\cvr w$ if $w$ covers $y$, i.e.,
if $y<w$ and $\ell(w)=\ell(y)+1$. (We refer to \cite{MR2133266} for basic properties about the Bruhat order.)
Let
\[
\smsqr=\{1,\dots,n-1\}\times\{1,\dots,n-1\},\ \ \bsqr=\{0,\dots,n\}\times\{0,\dots,n\}
\]
be the \emph{restricted square} and the \emph{framed square} respectively.
For any $w\in S_n$ let $\grph_w=\{\pnt_w(i):i=1,\dots,n\}\subset\bsqr$ be the graph of $w$ where $\pnt_w(i)=(i,w(i))$
and let $\rk_w:\bsqr\rightarrow\Z_{\ge0}$ be the \emph{rank function}
\[
\rk_w(i,j)=\#\{u=1,\dots,i:w(u)\le j\}.
\]
Note that $\rk_w(0,i)=\rk_w(i,0)=0$ and $\rk_w(n,i)=\rk_w(i,n)=i$ for all $i$.
Recall that for any $w,y\in S_n$ we have $y\le w$ if and only if $\rk_w\le\rk_y$ on
$\bsqr$ (or equivalently, on $\smsqr$).

For any $p=(i,j),p'=(i',j')\in\bsqr$ we write
\[
\crnrs(p,p')=\{(i,j'),(i',j)\}
\]
and define the \emph{difference function}
\begin{equation} \label{def: diff}
\diff_w(p,p'):=\sum_{q\in\{p,p'\}}\rk_w(q)-\sum_{q\in\crnrs(p,p')}\rk_w(q).
\end{equation}
We also write $p'>p$ (resp., $p'\ge p$) if $i'>i$ and $j'>j$ (resp., $i'\ge i$ and $j'\ge j$).
(We caution that $<$ is not the strict partial order subordinate to $\le$.)
Clearly, $\diff_w(p,p')=\diff_w(p',p)$.
If $p'\ge p$ then
\[
\diff_w(p,p')=\#\{i<u\le i':j<w(u)\le j'\}=\#(\grph_w\cap (p,p'])\ge0
\]
where
\[
(p,p']=\{q\in\bsqr:p<q\le p'\}.
\]
Similarly we also use the notation $[p,p')=\{q\in\bsqr:p\le q<p'\}$ and $(p,p')=\{q\in\smsqr:p<q<p'\}$.
We write $p\cmpr p'$ if either $p>p'$ or $p'>p$; equivalently $(i'-i)(j'-j)>0$.

Define the set of \emph{pairs}
\[
\pairset=\{(w,x)\in S_n\times S_n:x\le w\}.
\]
For any $\pair=(w,x)\in\pairset$ we write
\[
\rk_\pair=\rk_x-\rk_w:\bsqr\rightarrow\Z_{\ge0}
\]
and for any $p,p'\in\bsqr$ let
\[
\diff_\pair(p,p')=\diff_x(p,p')-\diff_w(p,p')=\sum_{q\in\{p,p'\}}\rk_\pair(q)-\sum_{q\in\crnrs(p,p')}\rk_\pair(q).
\]
Define the \emph{level set} of $\pair$ and its complement to be
\begin{equation} \label{def: X}
\X_\pair=\{p\in\bsqr:\rk_\pair(p)=0\},\ \ \Xc_\pair=\{p\in\bsqr:\rk_\pair(p)>0\}\subset\smsqr.
\end{equation}
Note that if $x\le y\le w$ then
\[
\X_\pair=\X_{(w,y)}\cap\X_{(y,x)}.
\]
We will often use the following simple fact:
\begin{equation} \label{sf}
\text{if $p,q\in\X_\pair$, $p\cmpr q$ and $\diff_w(p,q)=0$ then $\diff_x(p,q)=0$ and $\crnrs(p,q)\subset\X_\pair$.}
\end{equation}
Indeed, $0\le\diff_x(p,q)=\diff_\pair(p,q)=-\sum_{p'\in\crnrs(p,q)}\rk_\pair(p')\le0$.

Denote by $\Trans$ the set of transpositions in $S_n$.
Thus, $\Trans=\{t_{i,j}:1\le i<j\le n\}$ where $t_{i,j}=t_{\{i,j\}}$ is the transposition interchanging $i$ and $j$.
By our convention, when using the notation $t_{i,j}$ (in contrast to $t_{\{i,j\}}$) it is understood, often implicitly, that $i<j$.
Note that $w<wt_{i,j}$ if and only if $w(i)<w(j)$ if and only if $\pnt_w(i)<\pnt_w(j)$ and in this case
\[
\Xc_{(wt_{i,j},w)}=[\pnt_w(i),\pnt_w(j));
\]
otherwise $wt_{i,j}<w$. Thus,
\[
\ell(w)=\#\{t\in\Trans:wt<w\}.
\]

For $\pair=(w,x)\in\pairset$ let $\ell(\pair)=\ell(w)-\ell(x)\ge 0$,
\[
\RR_\pair=\{t\in\Trans:x<xt\le w\},\ \
\RRs_\pair=\{t\in\Trans:x\cvr xt\le w\}.
\]
Thus, $\ell(\pair)=0$ if and only if $w=x$ if and only if $\RRs_\pair=\emptyset$.
We have
\begin{equation} \label{eq: tijR}
t_{i,j}\in\RR_\pair\iff\pnt_x(i)<\pnt_x(j)\text{ and }\Xc_{(xt_{i,j},x)}=[\pnt_x(i),\pnt_x(j))\subset\Xc_\pair.
\end{equation}
Hence,
\begin{equation} \label{eq: 1234}
\text{if $i<j<k<l$, $x(j)<x(i)<x(l)<x(k)$ and $t_{i,k},t_{j,l}\in\RR_\pair$ then $t_{i,l}\in\RR_\pair$}.
\end{equation}
If $t=t_{i,j}\in\RR_\pair$ then $t\in\RRs_\pair$ if and only if there does not exist $i'$ such that
$t_{i,i'},t_{i',j}\in\RR_\pair$.

For any $t\in\RR_\pair$ let $\pair_t=(w,xt)\in\pairset$. Thus, $\ell(\pair_t)<\ell(\pair)$.

We have (see \S\ref{sec: gash})
\begin{equation} \label{eq: basicne}
\#\RR_\pair\ge\ell(\pair).
\end{equation}
We say that $\pair$ is \emph{smooth} if equality holds.

The geometric interpretation is as follows. As in the introduction, let $Y_w$ be the Schubert cell pertaining to $w$ in the flag variety of $\GL_n(\C)$
and let $X_w$ be the corresponding Schubert variety, i.e., the Zariski closure of $Y_w$.
Then $\pair=(w,x)\in\pairset$
if and only if $Y_x\subset X_w$ and in this case
\[
\#\RR_\pair+\ell(x)=\#\{t\in\Trans:xt\le w\}
\]
is the dimension of the tangent space of $X_w$ at any point of $Y_x$ \cite{MR752799}.
Thus, $\pair$ is smooth if and only if $Y_x$ is contained in the smooth locus of $X_w$.
(Other equivalent conditions are that $X_w$ is rationally smooth at any point of $Y_x$
or that the Kazhdan--Lusztig polynomial $P_{x,w}$ with respect to $S_n$ is $1$ \cite{MR560412, MR573434, MR788771}.
We refer the reader to \cite{MR1782635} for more details and generalizations.)
We also recall that $X_w$ is smooth (i.e., $(w,e)$ is smooth) if and only if $w$ is $3412$ and $4231$ avoiding \cite{MR1051089}.

\subsection{Symmetries}
For any $w\in S_n$ we write $w^*=w_0ww_0\in S_n$ for the \emph{upended} permutation where $w_0$ is the longest element of $S_n$.
For any $\pair=(w,x)\in\pairset$ let $\pair^*=(w^*,x^*)\in\pairset$ and $\pair^{-1}=(w^{-1},s^{-1})\in\pairset$
be the upended and \emph{inverted} pair respectively.
For any $p=(i,j)\in\bsqr$ let $p^*=(n-i,n-j)\in\bsqr$ and $p^{-1}=(j,i)\in\bsqr$ be the upended and inverted point respectively.
Thus, $\pnt_{w^*}(n+1-r)=\pnt_w(r)^*-(1,1)$ and $\pnt_{w^{-1}}(w(r))=\pnt_w(r)^{-1}$ for any $r$, $\grph_{w^*}=\grph_w^*-(1,1)$,
$\grph_{w^{-1}}=\grph_w^{-1}$, $\rk_{w^*}(p^*)=n-i-j+\rk_w(p)$, $\rk_{\pair^*}(p^*)=\rk_\pair(p)$,
$\rk_{w^{-1}}(p^{-1})=\rk_w(p)$ and $\rk_{\pair^{-1}}(p^{-1})=\rk_\pair(p)$,
In particular, $\X_{\pair^*}=\X_\pair^*$ and $\X_{\pair^{-1}}=\X_\pair^{-1}$.
Note that $\RR_{\pair^*}=\RR_\pair^*$ and $\RR_{\pair^{-1}}=\RR_\pair^x$ where $t^x=xtx^{-1}$; similarly for $\RRs$.
For any $t\in\RR_\pair$ we have $(\pair_t)^*=(\pair^*)_{t^*}$ and $(\pair_t)^{-1}=(\pair^{-1})_{t^x}$.

\subsection{Statement}
We make the following key definition.
\begin{definition}
Let $w\in S_n$.
\begin{enumerate}
\item A subset $A\subset\bsqr$ is called \emph{tight} with respect to $w$ if
for any $y\in S_n$ such that $\rk_y\rest_A\equiv\rk_w\rest_A$ we have $y\le w$.
\item We say that a pair $\pair=(w,x)\in\pairset$ is tight if $\X_\pair$ is tight with respect to $w$.
\end{enumerate}
\end{definition}

The main result of the paper is the following equivalent reformulation of Theorem \ref{thm: mainintro}.
\begin{theorem} \label{thm: main}
Every smooth pair $\pair=(w,x)\in\pairset$ is tight.
\end{theorem}

\begin{remark}
In the case where $x=e$, Theorem \ref{thm: main} follows from \cite{MR870962, MR1013667}.
Moreover, it was proved by Gasharov--Reiner \cite{MR1934291} that $(w,e)$ is tight
(namely, $X_w$ is defined by inclusions)
if and only if $w$ avoids the patterns $4231$, $35142$, $42513$ and $351624$.
Curiously, this condition occurs in other contexts as well (cf. \cite{MR3644818}).
\end{remark}

\begin{remark}
In \cite{MR1154177} Fulton defined the \emph{essential set} of $w$ by
\[
\ess(w)=\{(i,j)\in\smsqr:w(i)\le j<w(i+1)\text{ and }w^{-1}(j)\le i<w^{-1}(j+1)\}.
\]
He showed, among other things, that
\[
y\le w\iff\rk_y(p)\ge\rk_w(p)\text{ for all }p\in\ess(w)
\]
and that this equivalence ceases to be true if we replace $\ess(w)$ by any proper subset.
See \cite{MR1412437} for a follow-up on these ideas.

Thus, if $\X_\pair\supset\ess(w)$ then $\pair$ is tight. If $x=e$ the converse is also true
\cite{MR1934291}. However, in general this is not the case.
For instance, for $n=5$ and the smooth pair $\pair=(35142,21345)$ we have
\[
\ess(w)=\{(1,3),(3,1),(3,3)\}\not\subset\X_\pair\cap\smsqr=\{(1,1),(1,3),(1,4),(3,1),(4,1)\}.
\]
This is one of the reasons why the present proof of Theorem \ref{thm: main} is technically more complicated than that of
\cite[Theorem 4.2]{MR1934291}.
\end{remark}

\subsection{First step}
Let $w\in S_n$.
For any subsets $A,B\subset\bsqr$ define
\[
C_A^w(B)=B\cup\{p\in\bsqr:\exists p'\in A\text{ such that }p\cmpr p',\diff_w(p,p')=0\text{ and }\crnrs(p,p')\subset B\}.
\]
The following simple lemma will be crucial for the argument.
\begin{lemma}
Let $w,y\in S_n$ and $A,B\subset\bsqr$.
Suppose that $\rk_y\rest_A\equiv\rk_w\rest_A$ and $\rk_y\ge\rk_w$ on $B$.
Then $\rk_y\ge\rk_w$ on $C_A^w(B)$.
\end{lemma}

\begin{proof}
Let $p\in C_A^w(B)$. If $p\in B$ there is nothing to prove.
Otherwise, let $p'\in A$ be such that $p'\cmpr p$, $\diff_w(p,p')=0$ and $\crnrs(p,p')\subset B$.
Then,
\[
\rk_y(p)+\rk_y(p')\ge\sum_{q\in\crnrs(p,p')}\rk_y(q)\ge\sum_{q\in\crnrs(p,p')}\rk_w(q)
=\rk_w(p)+\rk_w(p')=\rk_w(p)+\rk_y(p').
\]
Hence, $\rk_y(p)\ge\rk_w(p)$ as required.
\end{proof}

For any subset $A\subset\bsqr$ we define inductively $A_0=A$ and $A_k=C_A^w(A_{k-1})$, $k>0$.
The set $\infl^w(A)=\cup_{k\ge0}A_k$ is called the \emph{influence set} of $A$.

Clearly, if $A'\supset A$ then $\infl^w(A')\supset\infl^w(A)$.

\begin{corollary} \label{cor: scor}
Let $w\in S_n$.
For any $y\in S_n$ we have $\rk_y\ge\rk_w$ on
\[
\infl^w(\{(i,j)\in\bsqr:\rk_w(i,j)=\rk_y(i,j)\}).
\]
\end{corollary}

\begin{definition}
We say that $\pair=(w,x)\in\pairset$ is \emph{influential} if $\infl^w(\X_\pair)=\bsqr$.
\end{definition}
By Corollary \ref{cor: scor}, every influential pair is a tight pair.

We will in fact prove the following stronger version of Theorem \ref{thm: main}.
\begin{theorem} \label{thm: main'}
Every smooth pair is influential.
\end{theorem}

\subsection{Statement of main technical result} \label{sec: defbase}
In order to prove Theorem \ref{thm: main'} we introduce more notation and terminology.

\begin{definition}
Let $\pair=(w,x)\in\pairset$ and $p=(i,j)\in\smsqr$.
\begin{enumerate}
\item Define $\nextd{p}_\pair=(i',j)\in\bsqr$ by
\[
i'=\max\{l\ge i:(l,j')\in\X_\pair\text{ and }\diff_w(p,(l,j'))=0\}\text{ where }
j'=\min\{k>j:(i,k)\in\X_\pair\}.
\]
Similarly, define $\nextu{p}_\pair=(i',j)\in\bsqr$ by
\[
i'=\min\{l\le i:(l,j')\in\X_\pair\text{ and }\diff_w(p,(l,j'))=0\}\text{ where }
j'=\max\{k<j:(i,k)\in\X_\pair\}.
\]
\item We say that the condition $\prpd_\pair(p)$ (resp., $\prpu_\pair(p)$) is satisfied if
$\rk_\pair(\nextd{p}_\pair)<\rk_\pair(p)$ (resp., $\rk_\pair(\nextu{p}_\pair)<\rk_\pair(p)$).
\item We write
\[
\qrp_\pair(p)=\prpd_\pair(p)\lor\prpu_\pair(p)\lor\prpu_\pair(\nextd{p}_\pair)\lor\prpd_\pair(\nextu{p}_\pair).
\]
Thus, by definition the condition $\qrp_\pair(p)$ holds if at least one of the conditions $\prpd_\pair(p)$, $\prpu_\pair(p)$,
$\prpu_\pair(\nextd{p}_\pair)$ or $\prpd_\pair(\nextu{p}_\pair)$ is satisfied.
\end{enumerate}
\end{definition}

Note that the conditions $\prpd_\pair(p)$ and $\prpu_{\pair^*}(p^*)$ are equivalent, as are the conditions $\qrp_\pair(p)$ and $\qrp_{\pair^*}(p^*)$.
However, the latter are not equivalent to $\qrp_{\pair^{-1}}(p^{-1})$.

The key technical result is the following.
\begin{proposition} \label{prop: main}
Suppose that $\pair=(w,x)\in\pairset$ is a smooth pair. Then $\qrp_\pair(p)$ holds for any $p\in\Xc_\pair$.
\end{proposition}
We will prove Proposition \ref{prop: main} in sections \ref{sec: bcase} and \ref{sec: propmain} below.
Before that, let us explain how Proposition \ref{prop: main} implies Theorem \ref{thm: main'} (and hence Theorem \ref{thm: main}).

Let $\pair=(w,x)\in\pairset$ be a smooth pair and let
\[
\X_\pair^{\le i}=\{p\in\bsqr:\rk_\pair(p)\le i\}, \ \ i=0,1,\dots,n.
\]
Clearly, $\X_\pair^{\le 0}=\X_\pair$ and $\X_\pair^{\le n}=\bsqr$.
We show by induction on $i$ that $\X_\pair^{\le i}\subset\infl^w(\X_\pair)$.
For the induction step, let $p\in\smsqr$ with $\rk_\pair(p)=i>0$. By Proposition \ref{prop: main} $\qrp_\pair(p)$ holds.
If $\prpd_\pair(p)$ or $\prpu_\pair(p)$ is satisfied then it immediately follows from the definitions that
$p\in C_{\X_\pair}^w(\X_\pair^{\le i-1})$.
Otherwise, $\prpu_\pair(\nextd{p}_\pair)$ or $\prpd_\pair(\nextu{p}_\pair)$ holds and then respectively
$\nextd{p}_\pair\in C_{\X_\pair}^w(\X_\pair^{\le i-1})$ or $\nextu{p}_\pair\in C_{\X_\pair}^w(\X_\pair^{\le i-1})$.
Either way $p\in C_{\X_\pair}^w(C_{\X_\pair}^w(\X_\pair^{\le i-1}))$ and the induction hypothesis implies that
$p\in\infl^w(\X_\pair)$ as required.

We end this section by commenting on the proof of Proposition \ref{prop: main}.
We first consider an important special case in \S\ref{sec: bcase} where a stronger statement holds.
It is proved by induction on $\ell(\pair)$.
It is here that we use the smoothness of $\pair$ in a crucial way. 
The general case is proved in \S\ref{sec: propmain}, also by induction on $\ell(\pair)$.
The induction step is facilitated by the special case considered in \S\ref{sec: bcase}.
However, the smoothness assumption is no longer used directly -- only through the induction hypothesis
and the appeal to the special case.
Unfortunately, the argument is complicated by the four-headed and non-symmetric nature of property $\qrp_\pair(p)$,
which leads to a lengthy case-by-case analysis.

\section{Basic case} \label{sec: bcase}
In this section we carry out the first step of the proof of Proposition \ref{prop: main}.
\subsection{Statement} \label{sec: statebase}
\begin{definition} \label{def: prp}
For any $\pair=(w,x)\in\pairset$ and $p=(i,j)\in\Xc_\pair$ let
\begin{align*}
\rdshft p_\pair&=(\min\{l>i:(l,j)\in \X_\pair\},\min\{l>j:(i,l)\in \X_\pair\})>p,\\
\lushft p_\pair&=(\max\{l<i:(l,j)\in \X_\pair\},\max\{l<j:(i,l)\in \X_\pair\})<p,
\end{align*}
so that $\crnrs(p,\rdshft p_\pair),\crnrs(p,\lushft p_\pair)\subset\X_\pair$.
We say that $\rshft\prp_\pair(p)$ (resp., $\lshft\prp_\pair(p)$) is satisfied if
\begin{equation} \label{eq: rkcond}
\diff_x(p,\rdshft p_\pair)=1\text{ (resp., }\diff_x(p,\lushft p_\pair)=1).
\end{equation}
Finally, we write
\[
\prp_\pair(p)=\rshft\prp_\pair(p)\lor\lshft\prp_\pair(p).
\]
\end{definition}

The condition $\rshft\prp_\pair(p)$ turns out to be stronger than the condition $\prpd_\pair(p)$ defined in \S\ref{sec: defbase}
(see Lemma \ref{lem: prpdp} below).
At any rate, we will not use the notation of  \S\ref{sec: defbase} in this section.

Let $p\in\Xc_\pair$ and $p'=\rdshft p_\pair$. Since $\crnrs(p,p')\subset\X_\pair$, we have
\[
\diff_x(p,p')=\diff_w(p,p')+\rk_\pair(p)+\rk_\pair(p').
\]
Thus,
\begin{equation} \label{eq: equrpop}
\text{$\rshft\prp_\pair(p)$ holds if and only if $\rk_\pair(p)=1$, $\rdshft p_\pair\in\X_\pair$ and $\diff_w(p,\rdshft p_\pair)=0$.}
\end{equation}
Similarly for $\lshft\prp_\pair(p)$ where $\rdshft p_\pair$ is replaced by $\lushft p_\pair$.

Note that the conditions $\rshft\prp_\pair(p)$, $\lshft\prp_{\pair^*}(p^*)$ and $\rshft\prp_{\pair^{-1}}({p^{-1}})$ are equivalent.

Consider the \emph{critical set}
\begin{equation} \label{def: critset}
\crit_\pair=\bigcup_{t\in\RR_\pair}(\X_{\pair_t}\setminus \X_{(xt,x)})=
\bigcup_{t\in\RR_\pair}\X_{\pair_t}\setminus \X_\pair.
\end{equation}
Note that if $p\in\crit_\pair$ then $\rk_\pair(p)=1$. (However, the converse is not true in general, even if $\pair$ is smooth.)
Also, $\crit_{\pair^*}=\crit_\pair^*$ and $\crit_{\pair^{-1}}=\crit_\pair^{-1}$.

The main result of this section is the following.
\begin{proposition} \label{prop: main1'}
Let $\pair=(w,x)\in\pairset$ be a smooth pair.
Then for any $p\in\crit_\pair$:
\begin{enumerate}
\item \label{part: mainr} $\prp_\pair(p)$ is satisfied.
\item \label{part: suppl} If $\diff_w(p,\rdshft p_\pair)=0$ then $\rshft\prp_\pair(p)$ holds (or, equivalently, $\rdshft p_\pair\in\X_\pair$).
\item Similarly, if $\diff_w(p,\lushft p_\pair)=0$ then $\lshft\prp_\pair(p)$ holds (or, equivalently, $\lushft p_\pair\in\X_\pair$).
\end{enumerate}
\end{proposition}

In the rest of the section we will prove Proposition \ref{prop: main1'} by induction on $\ell(\pair)$.
The induction step reduces the statement to a special case which will be examined directly.
This reduction step uses the smoothness in a crucial way.
We first explain a purely formal reduction step.

\subsection{Preliminary reduction} \label{sec: softred}
We record some notation and results from \cite[\S5]{MR1990570} adjusted to the setup at hand.

Recall the following standard result (cf. proof of \cite[Theorem 2.1.5]{MR2133266}).
\begin{equation} \label{betii}
\text{Suppose that $\pair=(w,x)\in\pairset$ and $\pnt_x(i)\in\Xc_\pair$.
Then $\exists i'$ such that $t_{i,i'}\in\RR_\pair$.}
\end{equation}

For any subset $I\subset\{1,\dots,n\}$ of size $m$ let $\sigma_I:\{1,\dots,m\}\rightarrow I$ be the monotone bijection
and define $\pi_I:\{0,1,\dots,n\}\rightarrow\{0,1,\dots,m\}$
by $\pi_I(i)=\max\{j:\sigma_I(j)\le i\}$ (with $\max\emptyset=0$), a weakly monotone function.

Fix $\pair=(w,x)\in\pairset$ and let
\[
I=\{i=1,\dots,n:\exists i'\ne i\text{ such that }t_{\{i,i'\}}\in\RR_\pair\}
\]
and $m=\# I$. We say that $\pair$ is \emph{reduced} if $I=\{1,\dots,n\}$.

By \cite[Proposition 14]{MR1990570} and \eqref{betii} we have $\pnt_x(i)=\pnt_w(i)\in\X_\pair$ for all $i\notin I$ and in particular $J:=x(I)=w(I)$.
Let $\trnc\pair=(\pi_J\circ w\circ\sigma_I,\pi_J\circ x\circ\sigma_I)\in S_m\times S_m$ be the \emph{reduction} of $\pair$.
For any $t\in\Trans$ let $\trnc t\in\RR_m\cup\{e\}$ be given by
$\trnc{t_{i,j}}=t_{\pi_I(i),\pi_I(j)}$ if $\pi_I(i)\ne\pi_I(j)$ and $\trnc{t_{i,j}}=e$ otherwise.
Finally, for any $p=(i,j)\in\bsqr$ we write $\trnc p=(\pi_I(i),\pi_J(j))\in\bsqro_m$. (Of course,
$\trnc t$ and $\trnc p$ depend implicitly on $\pair$.)

The following elementary result is essentially in \cite[\S 5]{MR1990570}. We omit the details.
\begin{lemma} \label{lem: elem}
Let $\pair=(w,x)\in\pairset$. Then
\begin{enumerate}
\item $\trnc\pair\in\pairsets_m$.
\item $\ell(\trnc\pair)=\ell(\pair)$.
\item For any $p\in\bsqr$ we have $\rk_\pair(p)=\rk_{\trnc\pair}(\trnc p)$.
In particular, $p\in\X_\pair$ if and only if $\trnc p\in\X_{\trnc\pair}$.
\item For any $p_1,p_2\in\bsqr$ with $p_1\le p_2$ we have $\diff_w(p_1,p_2)\ge\diff_{\pi_J\circ w\circ\sigma_I}(\trnc p_1,\trnc p_2)$
and $\diff_\pair(p_1,p_2)=\diff_{\trnc\pair}(\trnc p_1,\trnc p_2)$.
\item The map $t\mapsto\trnc t$ defines a bijection between $\RR_\pair$ and $\RR_{\trnc\pair}$.
\item For any $t\in\RR_\pair$ and $p\in\bsqr$ we have $p\in\X_{\pair_t}$ if and only if $\trnc p\in\X_{\trnc\pair_{\trnc t}}$.
\item $\trnc\pair$ is smooth if and only if $\pair$ is smooth.
\item $\trnc\pair$ is reduced.
\item $p\in \crit_\pair$ if and only if $\trnc p\in\crit_{\trnc\pair}$.
\item For any $p\in\Xc_\pair$ we have $\rdshft{\trnc p}_{\trnc\pair}=\trnc{\rdshft p_\pair}$ and
$\lushft{\trnc p}_{\trnc\pair}=\trnc{\lushft p_\pair}$.
\end{enumerate}
\end{lemma}

\begin{lemma} \label{lem: redred}
Let $\pair=(w,x)\in\pairset$ and $p\in\Xc_\pair$ with $\rk_\pair(p)=1$.
Then
\begin{gather*}
\rshft\prp_\pair(p)\iff\rshft\prp_{\trnc\pair}(\trnc p),\ \ \ \lshft\prp_\pair(p)\iff\lshft\prp_{\trnc\pair}(\trnc p)\\
\diff_w(p,\rdshft p_\pair)=0\iff\diff_w(\trnc p,\rdshft{\trnc p}_{\trnc\pair})=0,\ \ \
\diff_w(p,\lushft p_\pair)=0\iff\diff_w(\trnc p,\lushft{\trnc p}_{\trnc\pair})=0.
\end{gather*}
Thus Proposition \ref{prop: main1'} holds for $\trnc\pair$ and $\trnc p$ if and only if it holds for $\pair$ and $p$.
\end{lemma}

\begin{proof}
Let $p=(i,j)$, $p'=\rdshft p_\pair=(i',j')$, $\trnc\pair=(w',x')$.
Suppose that $\diff_{w'}(\trnc p,\trnc p')=0$ but $\diff_w(p,p')>0$.
Then $\emptyset\ne\grph_w\cap (p,p']\subset\grph_x\cap\X_\pair$.
Let $q=\pnt_w(i_0)\in (p,p']$ (with $i_0\notin I$). Then
$\diff_w(p,q)\le\diff_x(p,q)$, i.e., $\diff_\pair(p,q)\ge0$.
On the other hand, $\rk_\pair(p)=1$ and $q\in\X_\pair$. Hence
$\crnrs(p,q)\cap\X_\pair\ne\emptyset$. However, if say $(i_0,j)\in\X_\pair$ then
$(i_0-1,j)\in\X_\pair$ and hence by the definition of $i'$ we would have $i'<i_0$ in contradiction to the fact that $q\in (p,p']$.
By \eqref{eq: equrpop} we also infer the equivalence of $\rshft\prp_\pair(p)$ and $\rshft\prp_{\trnc\pair}(\trnc p)$.
By symmetry we get the other equivalences.
\end{proof}

Thus, it suffices to prove Proposition \ref{prop: main1'} in the case where $\pair$ is reduced (and smooth).

We make the following simple remark.
\begin{equation} \label{eq: redec}
\text{If $\pair$ is reduced and $x$ is monotone decreasing for $i<i_1$ then $w(i)>x(i)$ for all $i<i_1$.}
\end{equation}
Otherwise, if $w(i)\le x(i)$ then $\rk_w(\pnt_w(i))\ge 1\ge\rk_x(\pnt_w(i))$ by the assumption on $x$.
Hence $w(i)=x(i)$ and $\pnt_w(i)\in\X_\pair$. But then $t_{i,i'}\notin\RR_\pair$ for all $i'$ by \eqref{eq: tijR}.
Since also $t_{i',i}\notin\RR_\pair$ for all $i'$ (by the assumption on $x$) $\pair$ cannot be reduced.

\subsection{Gasharov's map} \label{sec: gash}
Let $\pair=(w,x)\in\pairset$ and $t'\in\RR_\pair$.
In \cite{MR1827861} Gasharov introduced an injective map
\[
\phi^\pair_{t'}:\RR_{\pair_{t'}}\rightarrow\RR_\pair\setminus\{t'\}.
\]
We will use the following variant of \cite{MR1990570}:\footnote{It [loc. cit.] this is defined on a case-by-case basis,
but it amounts to the same formula as given here}
\[
\phi^\pair_{t'}(t)=\begin{cases}t^{t'}&\text{if }t^{t'}\in\RR_\pair,\\t&\text{otherwise.}\end{cases}
\]
The map $\phi^\pair_{t'}$ is well-defined and injective (\cite[Theorem 20]{MR1990570}).
Taking $t'\in\RRs_\pair$, this gives a combinatorial inductive proof of the inequality \eqref{eq: basicne}.
It also yields that
\[
\pair\text{ is smooth $\iff \phi^\pair_{t'}$ is surjective and $\pair_{t'}$ is smooth}.
\]
Hence, if $\pair$ is smooth then $(w,x')$ is smooth for any $x\le x'\le w$.
(Again, this fact is also clear from the geometric characterization.)

\begin{corollary} \label{cor: smth}
Suppose that $\pair=(w,x)$ is smooth, 
$t'\in\RRs_\pair$ and $t\in\RR_\pair$ with $t\ne t'$. Then at least one of $t$ or $t^{t'}$ is in $\RR_{\pair_{t'}}$.
\end{corollary}

\subsection{Main reduction} \label{sec: smoothred}
\begin{definition}
Let $\pair=(w,x)\in\pairset$ and $t=t_{i_1,i_2}\in\RR_\pair$. We say that $\pair$ is \emph{$t$-minimal} if
for every $t_{i_1',i_2'}\in\RR_\pair$ we have $i_1=i_1'$ or $i_2'=i_2$.
We say that $\pair$ is minimal if it is $t$-minimal for some $t\in\RR_\pair$.
\end{definition}

Note that it is not excluded that $\pair$ is both $t_1$-minimal and $t_2$-minimal for two different
$t_1,t_2\in\RR_\pair$.

Also note that by Lemma \ref{lem: elem} with its notation $\pair$ is $t$-minimal if and only if $\trnc\pair$ is $\trnc t$-minimal.
In particular, $\pair$ is minimal if and only if $\trnc\pair$ is.

\begin{definition}
Let $\pair=(w,x)\in\pair$ be a smooth pair. For any $p\in\Xc_\pair$ let
\[
\RR_\pair(p)=\{t\in\RR_\pair:p\in\X_{\pair_t}\}
\]
so that $p\in\crit_\pair$ if and only if $\RR_\pair(p)\ne\emptyset$. 
Let
\[
\crit^{\red}_\pair=\{p\in \crit_\pair:\pair\text{ is $t$-minimal for every }t\in\RR_\pair(p)\}.
\]
\end{definition}

Note that by Lemma \ref{lem: elem}, given $p\in\Xc_\pair$ and $t\in\RR_\pair$ we have $t\in\RR_\pair(p)$ if and only if
$\trnc t\in\RR_{\trnc p}(\trnc\pair)$. Hence,
\begin{equation} \label{eq: redcrit}
p\in\crit^{\red}_\pair\text{ if and only if }\trnc p\in\crit^{\red}_{\trnc\pair}.
\end{equation}

We will prove below that
\begin{equation} \label{ass: spclcase}
\text{Proposition \ref{prop: main1'} holds for every $\pair=(w,x)\in\pairset$ and $p\in\crit^{\red}_\pair$.}
\end{equation}

Assuming this for the moment, let us prove Proposition \ref{prop: main1'} by induction on $\ell(\pair)$.
\begin{proof}[Proof of Proposition \ref{prop: main1'}]
Let $\pair=(w,x)\in\pairset$ be a smooth pair and $p\in\crit_\pair$.
If $p\in\crit_\pair^{\red}$ then Proposition \ref{prop: main1'} holds by assumption.
Otherwise, there exists $t=t_{i_1,i_2}\in\RR_\pair(p)$ and
$t'=t_{i_1',i_2'}\in\RR_\pair$ such that $i_1'\ne i_1$ and $i_2'\ne i_2$.

Assume first that either $i_1'=i_2$ or $i_2'=i_1$.
Upon upending $\pair$, $p$, $t$ and $t'$ 
if necessary we may assume without loss of generality that $i_1'=i_2$.
Also, by changing $t'$ if necessary (without changing $i_1'$), we may assume that $t'\in\RRs_\pair$.
Note that by our condition on $t$, $t'$ and $p$ we have $\rdshft p_{\pair_{t'}}=\rdshft p_\pair$ and
$\lushft p_{\pair_{t'}}=\lushft p_\pair$.
In particular,
\[
\diff_w(p,\rdshft p_{\pair_{t'}})=0\iff\diff_w(p,\rdshft p_\pair)=0\text{ and }
\diff_w(p,\lushft p_{\pair_{t'}})=0\iff\diff_w(p,\lushft p_\pair)=0.
\]
Also, the conditions $\lshft\prp_{\pair_{t'}}(p)$ and $\lshft\prp_\pair(p)$ are clearly equivalent.
Thus, to conclude Proposition \ref{prop: main1'} for $(\pair,p)$
we will show that by the induction hypothesis Proposition \ref{prop: main1'} holds for $(\pair_{t'},p)$
and that the conditions $\rshft\prp_{\pair_{t'}}(p)$ and $\rshft\prp_\pair(p)$ are equivalent.
It is at this point that we use the smoothness of $\pair$. Namely, by Corollary \ref{cor: smth} we have
$t\in\RR_{\pair_{t'}}$ or $t^{t'}\in\RR_{\pair_{t'}}$.
Upon inverting $\pair$ and $p$ and conjugating $t$ and $t'$ by $x$ 
if necessary we may assume without loss of generality that $t\in\RR_{\pair_{t'}}$.
(Note that $t^{t'}\in\RR_{\pair_{t'}}$ if and only if $t^x\in\RR_{(\pair_{t'})^{-1}}=\RR_{(\pair^{-1})_{t'^x}}$.)
Clearly, $p\in \X_{(\pair_{t'})_t}\setminus \X_{\pair_{t'}}$.
Thus, we may apply the induction hypothesis to $\pair_{t'}$.
Let $p'=\rdshft p_\pair=\rdshft p_{\pair_{t'}}=(i',j')$.
Since $t\in\RR_{\pair_{t'}}$ we must have $j'\ge x(i_2')$.
Hence, $p'\in\X_{(xt',x)}$ and thus $p'\in\X_{\pair_{t'}}\iff p'\in\X_\pair$.
By \eqref{eq: equrpop}
the conditions $\rshft\prp_{\pair_{t'}}(p)$ and $\rshft\prp_\pair(p)$ are equivalent as required.

By \eqref{betii} and the case considered above, we may assume for the rest of the proof that $q:=\pnt_x(i_2)\in\X_\pair$
(and symmetrically $\pnt_x(i_1)-(1,1)\in\X_\pair$).

We now show that $\diff_w(p,\rdshft{p}_\pair)=0$ implies $\rshft\prp_\pair(p)$.
In fact, we show that $\diff_w(p,q)=0$ implies $\rshft\prp_\pair(p)$. (Note that $q\le\rdshft p_\pair$.)
Indeed, if $\diff_w(p,q)=0$ then by \eqref{sf} (applied to $\pair_t$) $\crnrs(p,q)\subset\X_{\pair_t}\cap\X_{(xt,x)}=\X_\pair$
and $\diff_x(p,q)=1$. Hence, from the definition, $\rdshft p_\pair\le q$.
Thus, $\rdshft p_\pair=q$ and $\rshft\prp_\pair(p)$ is satisfied.

By a similar reasoning (or by symmetry) $\diff_w(p,\lushft{p}_\pair)=0$ implies $\lshft\prp_\pair(p)$.

It remains to show that $\prp_\pair(p)$ is satisfied in the case where $t$ and $t'$ commute (and $t'\ne t$).
Once again, by changing $t'$ if necessary, and using the previous case, we may assume that $t'\in\RRs_\pair$.
Then by Corollary \ref{cor: smth} $t\in\RR_{\pair_{t'}}$. Note that $\Xc_{(xt,x)}=\Xc_{(xt't,xt')}$ since $t$ and $t'$ commute.
Hence
\[
p\in\X_{\pair_t}\cap\Xc_{(xt,x)}\subset\X_{(\pair_{t'})_t}\cap\Xc_{(xt't,xt')}=\X_{(\pair_{t'})_t}\setminus\X_{\pair_{t'}}.
\]
Thus, by induction hypothesis $\prp_{\pair_{t'}}(p)$ holds.
By symmetry, we can assume without loss of generality that $\rshft\prp_{\pair_{t'}}(p)$ is satisfied.
Hence, $\diff_w(p,\rdshft p_{\pair_{t'}})=0$.
On the other hand, since $t\in\RR_{\pair_{t'}}$ we have $q\le\rdshft p_{\pair_{t'}}$.
Therefore $\diff_w(p,q)=0$. By the previous paragraph this implies $\rshft\prp_\pair(p)$.
\end{proof}

\subsection{Minimal reduced case}
It remains to consider the minimal reduced case, which we can explicate as follows.
\begin{lemma} \label{lem: mindredcase}
Suppose that $\pair=(w,x)\in\pairset$ is reduced and $t$-minimal.
Write $t=t_{i_1,i_2}$ and $t^x=t_{j_1,j_2}$.
Then $j_2-j_1+i_2-i_1\ge n$ and $\pair$ is given by
\begin{subequations} \label{eq: spclpair}
\begin{equation} \label{eq: shapes}
w(i)=\begin{cases}j_2+1-i&i<i_1,\\n+i_1-i&i_1\le i<n+i_1-j_2,\\n+1-i&n+i_1-j_2\le i\le i_2-j_1+1,\\
i_2+1-i&i_2-j_1+1<i\le i_2,\\n+j_1-i&i>i_2.\end{cases}
\end{equation}
\begin{equation} \label{eq: shapes0}
x(i)=\begin{cases}j_2-i&i<i_1,\\j_1&i=i_1,\\n+i_1+1-i&i_1<i\le n+i_1-j_2,\\
n+1-i&n+i_1-j_2<i\le i_2-j_1,\\i_2-i&i_2-j_1<i<i_2,\\j_2&i=i_2,\\n+j_1+1-i&i>i_2,\end{cases}
\end{equation}
\end{subequations}
\end{lemma}

\begin{proof}
Let $t=t_{i_1,i_2}$.
We proceed with the following steps, first dealing with $x$.
\begin{enumerate}
\item $t_{i,i_2}\in\RR_\pair$ for every $i<i_1$. In particular, $x(i)<j_2$ for all $i<i_1$.

Since $\pair$ is $t$-minimal and $i<i_1$ we cannot have $t_{i',i}\in\RR_\pair$ for any $i'$.
Therefore, since $\pair$ is reduced, there exists $i'$ such that $t_{i,i'}\in\RR_\pair$.
Since $\pair$ is $t$-minimal, $i'=i_2$.

\item $x(i)>x(i+1)$ for all $i<i_1$.

Otherwise, the relation $t_{i,i_2}\in\RR_\pair$ would imply $t_{i,i+1}\in\RR_\pair$, contradicting the assumption on $\pair$.

\item In a similar vein, for all $i>i_2$ we have $t_{i_1,i}\in\RR_\pair$, $j_1<x(i)$ and $x(i)<x(i-1)$.

\item For any $i<i_1$ and $i'>i_2$ we have $x(i)>x(i')$.

Otherwise, the relations $t_{i,i_2},t_{i_1,i'}\in\RR_\pair$ would give $t_{i,i'}\in\RR_\pair$ by \eqref{eq: 1234} which denies the assumption on $\pair$.

\item By passing to $\pair^{-1}$ we similarly have $t_{x^{-1}(j),i_2}\in\RR_\pair$ and $x^{-1}(j)>x^{-1}(j+1)$
for all $j<j_1$ while $t_{i_1,x^{-1}(j)}\in\RR_\pair$ and $x^{-1}(j)<x^{-1}(j-1)$ for all $j>j_2$.
Moreover, $x^{-1}(j)>x^{-1}(j')$ for all $j<j_1$ and $j'>j_2$.

\item Suppose that $i_1<i<i_2$ and $j_1<x(i)<j_2$. Then
\begin{enumerate}
\item $x(i')>x(i)$ for all $i'<i_1$.

Otherwise, $t_{i',i}\in\RR_\pair$ (since $t_{i',i_2}\in\RR_\pair$) in violation of the $t$-minimality of $\pair$.

\item Similarly, $x(i')<x(i)$ for all $i'>i_2$ and $x^{-1}(j')<i$
(resp., $x^{-1}(j')>i$) for all $j'>j_2$ (resp., $j'<j_1$).
\end{enumerate}

\item Suppose that $i_1<i<i'<i_2$ and $j_1<x(i),x(i')<j_2$. Then $x(i)>x(i')$.

Otherwise $t_{i,i'}\in\RR_\pair$, refuting the $t$-minimality of $\pair$.
\end{enumerate}
In conclusion, $n+i_1-j_2\le i_2-j_1$ and $x$ is given by \eqref{eq: shapes0}

Next we deal with $w$.
\begin{enumerate}
\item By \eqref{eq: redec} $w(i)>x(i)=j_2-i$ for all $i<i_1$.

\item $\rk_\pair(i_1-1,j)>0$ for all $j_2-i_1+1\le j<j_2$.

Indeed, $\rk_x(i_1-1,j)=j-j_2+i_1$ while $\rk_w(i_1-1,j)<j-j_2+i_1$ by the previous part.

\item $\rk_w(i_1-1,j_2)=i_1-1$.

Otherwise, we would have $\rk_\pair(i_1-1,j_2)>0$. In view of the previous part and the fact that
$t_{i_1-1,i_2}$ and $t_{i_1,n+i_1-j_2}$ are in $\RR_\pair$, this would imply that $t_{i_1-1,n+i_1-j_2}\in\RR_\pair$,
rebutting the $t$-minimality of $\pair$.

\item We conclude that $w(i)=x(i)+1$ for all $i<i_1$.

\item By symmetry $w(i)=x(i)-1$ for all $i>i_2$,
$w^{-1}(j)=x^{-1}(j)+1$ for all $j<j_1$ and $w^{-1}(j)=x^{-1}(j)-1$ for all $j>j_2$.

\item $w(n+i_1-j_2)=j_2-i_1+1$.

Let $j=w(n+i_1-j_2)$. We already have $j_1\le j\le j_2-i_1+1$. If $j\le j_2-i_1$ then $\rk_\pair(n+i_1-j_2,j)=0$,
gainsaying that $t_{i_1,i_2}\in\RR_\pair$.

\item Similarly $w(i_2-j_1+1)=n-i_2+j_1$.

\item $w(i)=x(i)$ for all $n+i_1-j_2<i\le i_2-j_1$.

This follows now again from the fact that $t_{i_1,i_2}\in\RR_\pair$.
\end{enumerate}
The result follows.
\end{proof}

The following assertion is straightforward.
\begin{lemma} \label{lem: basicbasic}
Let $1\le i_1<i_2\le n$ and $1\le j_1<j_2\le n$ be such that $j_2-j_1+i_2-i_1\ge n$ and let $\pair$ be given by
\eqref{eq: shapes}, \eqref{eq: shapes0}.\footnote{It can be easily shown that $\pair$ is smooth, reduced and $t_{i_1,i_2}$-minimal.
We will not use this fact explicitly.}
Then
\begin{enumerate}
\item If $i_1=j_1=1$ and $i_2+j_2=n+2$ then $w(i)=n+1-i$ for all $i$
and $x(i)=n+2-i$ for all $i>1$. We have
\[
\crit^{\red}_\pair=\{(i,j)\in\smsqr:i+j\le n\}
\]
and for any $p=(i,j)\in\crit^{\red}_\pair$
\begin{enumerate}
\item $\lushft p_\pair=(0,0)$.
\item $\lshft\prp_\pair(p)$ is satisfied.
\item $\rdshft p_\pair=(n+1-j,n+1-i)$.
\item $\rshft\prp_\pair(p)\iff\diff_w(p,\rdshft p_\pair)=0\iff i+j=n$.
\end{enumerate}
\item Similarly, if $i_2=j_2=n$ and $i_1+j_1=n$ then $w(i)=n+1-i$ for all $i$
and $x(i)=n-i$ for all $i<n$. We have
\[
\crit^{\red}_\pair=\{(i,j)\in\smsqr:n\le i+j\}
\]
and for any $p=(i,j)\in\crit^{\red}_\pair$
\begin{enumerate}
\item $\rdshft p_\pair=(n,n)$.
\item $\rshft\prp_\pair(p)$ is satisfied.
\item $\lushft p_\pair=(n-1-j,n-1-i)$.
\item $\lshft\prp_\pair(p)\iff\diff_w(p,\lushft p_\pair)=0\iff i+j=n$.
\end{enumerate}
\item In all other cases
\[
\crit^{\red}_\pair=\{(i,n-i):n+i_1-j_2\le i\le i_2-j_1\}
\]
and for any $p\in\crit^{\red}_\pair$ we have $\rdshft p_\pair=(i_2,j_2)$, $\lushft p_\pair=(i_1-1,j_1-1)$
and both conditions $\rshft\prp_\pair(p)$ and $\lshft\prp_\pair(p)$ are satisfied.
\end{enumerate}
\end{lemma}

\newcommand{\xpermcolor}{red}
\newcommand{\wpermcolor}{green}
\newcommand{\wxpermcolor}{blue}
\newcommand{\rkonecolor}{magenta}
\newcommand{\rktwocolor}{black}
\newcommand{\critsymb}{\text{C}}
\newcommand{\Xdot}{\color{\xpermcolor}\newmoon}
\newcommand{\Odot}{\color{\wpermcolor}\newmoon}
\newcommand{\Wdot}{\colorbox{\rkonecolor}{\color{\wxpermcolor}\newmoon}}
\newcommand{\Xdott}{\colorbox{\rkonecolor}{\Xdot}}
\newcommand{\Odott}{\colorbox{\rkonecolor}{\Odot}}
\newcommand{\Ydott}{\colorbox{\rkonecolor}{\color{\rkonecolor}\newmoon}}
\newcommand{\Ydottt}{\colorbox{\rktwocolor}{\color{\rktwocolor}\newmoon}}
\newcommand{\Sdot}{\colorbox{\rkonecolor}{\critsymb}}

\begin{example}
In the diagram below we draw the example for $n=12$, $i_1=4$, $j_1=3$, $i_2=10$, $j_2=11$.
The {\wpermcolor} (resp., \wpermcolor, \wxpermcolor) dots represent the points of $\grph_w\setminus\grph_x$ (resp., $\grph_x\setminus\grph_w$,
$\grph_x\cap\grph_w$). The background color of a point $p$ is white (resp., \rkonecolor, \rktwocolor) if $\rk_\pair(p)=0$ (resp., $1$, $2$).
The points in $\crit^{\red}_\pair$ are denotes by ``\critsymb''.

{
\setstretch{0}
\setlength\arraycolsep{0pt}
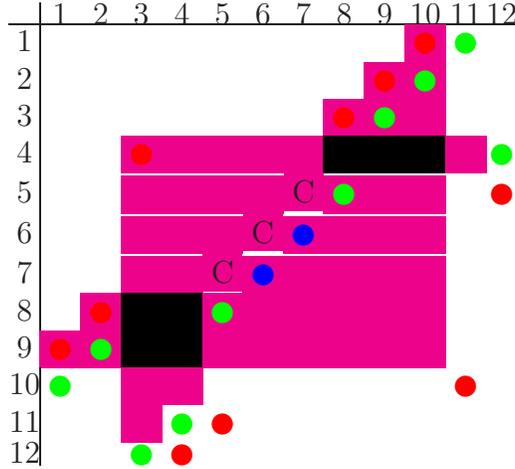
\[
\begin{array}{c|cccccccccccc}
& 1 & 2 & 3 & 4 & 5 & 6 & 7 & 8 & 9 & 10 & 11 & 12 \\
\hline
1 & & & & & & & & & & \Xdott & \Odot & \\
2 & & & & & & & & & \Xdott & \Odott & & \\
3 & & & & & & & & \Xdott & \Odott & \Ydott & & \\
4 & & & \Xdott & \Ydott & \Ydott & \Ydott & \Ydott & \Ydottt & \Ydottt & \Ydottt & \Ydott & \Odot \\
5 & & & \Ydott & \Ydott & \Ydott & \Ydott & \Sdot & \Odott & \Ydott & \Ydott & & \Xdot \\
6 & & & \Ydott & \Ydott & \Ydott & \Sdot & \Wdot & \Ydott & \Ydott & \Ydott & &   \\
7 & & & \Ydott & \Ydott & \Sdot & \Wdot & \Ydott & \Ydott & \Ydott & \Ydott & &   \\
8 & & \Xdott & \Ydottt & \Ydottt & \Odott & \Ydott & \Ydott & \Ydott & \Ydott & \Ydott & & \\
9 & \Xdott & \Odott & \Ydottt & \Ydottt & \Ydott & \Ydott & \Ydott & \Ydott & \Ydott & \Ydott & & \\
10 & \Odot & & \Ydott & \Ydott & & & & & & & \Xdot & \\
11 & & & \Ydott & \Odot & \Xdot & & & & & & & \\
12 & & & \Odot & \Xdot & & & & & & & &
\end{array}
\]
}
\end{example}

We now deduce assertion \eqref{ass: spclcase} to complete the proof of Proposition \ref{prop: main1'}.
By Lemma \ref{lem: redred} and \eqref{eq: redcrit} it is enough to prove \eqref{ass: spclcase} in the case where $\pair$ is reduced.
This case follows from Lemmas \ref{lem: mindredcase} and \ref{lem: basicbasic}.

\section{Proof of Proposition \ref{prop: main}} \label{sec: propmain}

\subsection{Notation and auxiliary results}
We go back to the conditions $\prpd_\pair(p)$, $\prpu_\pair(p)$, $\qrp_\pair(p)$ defined in \S\ref{sec: defbase}
and set some more notation.

Let $\pair=(w,x)\in\pairset$ and $p=(i,j)\in\smsqr$.
Recall that $\nextd{p}_\pair=(i',j)\in\bsqr$ was defined by
\[
i'=\max\{l\ge i:(l,j')\in\X_\pair\text{ and }\diff_w(p,(l,j'))=0\}\text{ where }
j'=\min\{k>j:(i,k)\in\X_\pair\}.
\]
We will also write $\rshft{p}_\pair=(i,j')\in\X_\pair$ and $\nextrd{p}_\pair=(i',j')\in\X_\pair$, so that $\diff_w(p,\nextrd{p}_\pair)=0$.

Similarly, define $\lshft{p}_\pair=(i,j')\in\X_\pair$ and $\nextlu{p}_\pair=(i',j')\in\X_\pair$ by
\[
j'=\max\{k<j:(i,k)\in\X_\pair\},\ \ i'=\min\{l\le i:(l,j')\in\X_\pair\text{ and }\diff_w(p,(l,j'))=0\}
\]
so that $\diff_w(p,\nextlu{p}_\pair)=0$ and recall that $\nextu{p}_\pair=(i',j)\in\bsqr$.

We caution that $\nextrd{p}_\pair$ and $\nextlu{p}_\pair$ should not be confused with $\rdshft p_\pair$ and $\lushft p_\pair$ defined in \S\ref{sec: statebase}.

Note that for any $p\in\Xc_\pair$ we have $\diff_x(p,\nextrd{p}_\pair)=\diff_\pair(p,\nextrd{p}_\pair)$.
Thus
\begin{equation} \label{prpdholds}
\prpd_\pair(p)\iff\diff_\pair(p,\nextrd{p}_\pair)>0\iff\diff_x(p,\nextrd{p}_\pair)>0.
\end{equation}

\begin{definition}
\begin{enumerate}
\item We say that $p$ is $\uparrow$-maximal (resp., $\downarrow$-maximal) with respect to $\pair$ if $\nextu{p}_\pair=p$
(resp., $\nextd{p}_\pair=p$).
\item We write
\[
\qrpd_\pair(p)=\prpd_\pair(p)\lor\prpu_\pair(\nextd{p}_\pair),\ \ \qrpu_\pair(p)=\prpu_\pair(p)\lor\prpd_\pair(\nextu{p}_\pair)
\]
so that
\[
\qrp_\pair(p)=\qrpd_\pair(p)\lor\qrpu_\pair(p).
\]
\end{enumerate}
\end{definition}

It is clear that
\begin{equation} \label{impluord}
\text{if $p$ is $\downarrow$-maximal (resp., $\uparrow$-maximal) then $\qrp_p(\pair)$ implies $\qrpu_p(\pair)$ (resp., $\qrpd_p(\pair)$).}
\end{equation}

Note that the conditions $\qrpd_\pair(p)$ and $\qrpu_{\pair^*}(p^*)$ are equivalent.

In the rest of this subsection we give some simple properties of the interplay betweem the various conditions
$\prpd_\pair(p)$, $\prpu_\pair(p)$, $\qrpd_\pair(p)$, $\qrpu_\pair(p)$, $\rshft\prp_\pair(p)$, $\lshft\prp_\pair(p)$.
They will be used in the induction step of Proposition \ref{prop: main}. Throughout let $\pair=(w,x)\in\pairset$.

\begin{lemma} \label{lem: prpdp}
Let $\rshft{p}_\pair=(i,j')$ (resp., $\lshft{p}_\pair=(i,j')$).
Then $\prpd_\pair(p)$ (resp., $\prpu_\pair(p)$) is satisfied if and only if there exists $i'>i$ (resp., $i'<i$) such that $p':=(i',j')\in\X_\pair$
and $\diff_w(p,p')=0<\diff_x(p,p')$. In particular, $\rshft\prp_\pair(p)$ (see Definition \ref{def: prp}) implies $\prpd_\pair(p)$ and similarly
$\lshft\prp_\pair(p)$ implies $\prpu_\pair(p)$.
\end{lemma}

\begin{proof}
The ``only if'' direction clear. For the ``if'' direction note that $p'':=\nextrd{p}_\pair=(i'',j')$ with $i''\ge i'$.
Therefore $\diff_\pair(p,p'')=\diff_x(p,p'')\ge\diff_x(p,p')=\diff_\pair(p,p')>0$ and hence $\rk_\pair(\nextd{p}_\pair)<\rk_\pair(p)$.
\end{proof}

\begin{remark}
In fact, it is easy to see that the condition $\rshft\prp_\pair(p)$ is equivalent to $\rk_\pair(p)=1$ and $\prpd_\pair(p)$.
We will not use this fact.
\end{remark}

\begin{lemma} \label{lem: nitm2}
Let $p=(i,j)\in\Xc_\pair$, $\nextlu{p}_\pair=(i',j')$ and $p'=\nextu{p}_\pair=(i',j)$.
Assume that $\prpu_\pair(p)$ is not satisfied. Then
\begin{enumerate}
\item $p'$ is $\uparrow$-maximal with respect to $\pair$.
\item \label{part: two} For any  $p''=(i_1,j)$ with $i'\le i_1\le i$ we have $\rk_\pair(p'')\ge\rk_\pair(p)$ and if equality holds then
$\nextu{p''}_\pair=p'$, $\lshft{p''}_\pair=(i_1,j')$ and $\prpu_\pair(p'')$ is not satisfied.
\item In particular, $\lshft{p'}_\pair=\nextlu{p}_\pair$.
\item If $\qrpu_\pair(p)$ is satisfied then $\nextd{p'}_\pair=(i'',j)$ with $i''>i$.
\end{enumerate}
\end{lemma}

\begin{proof}
Let $q=\nextlu{p}_\pair$ and $\lshft{p'}_\pair=(i',j_1)$. Then $j_1\ge j'$ since $q\in\X_\pair$. Assume on the contrary that $j_1>j'$.
Since $\prpu_\pair(p)$ is not satisfied, we have $\diff_x(q,p)=\diff_w(q,p)=0$.
Hence $\diff_x(q,(i,j_1))=\diff_w(q,(i,j_1))=0$ and since $(i',j'),(i,j'),(i',j_1)\in\X_\pair$ we would get $(i,j_1)\in\X_\pair$ which contradicts the definition of $j'$.
It is now clear that $p'$ is $\uparrow$-maximal.

Now let $p''=(i_1,j)$ with $i'\le i_1\le i$. Then $\diff_x((i_1,j'),p)=\diff_w((i_1,j'),p)=0$ and $\lshft p_\pair=(i,j')\in\X_\pair$
and hence $\rk_\pair(p'')=\rk_\pair(p)+\rk_\pair(i_1,j')$. Thus, $\rk_\pair(p'')\ge\rk_\pair(p)$ with an equality
if and only if $(i_1,j')\in\X_\pair$. In this case, by the same reasoning as before we have $\lshft{p''}_\pair=(i_1,j')$,
and hence $\nextu{p''}_\pair=p'$. In particular, $\prpu_\pair(p'')$ is not satisfied.

The last part is now clear from the definitions.
\end{proof}

For convenience we also write down the symmetric version of Lemma \ref{lem: nitm2}.
\begin{lemma} \label{lem: nitm}
Let $p=(i,j)\in\Xc_\pair$, $\nextrd{p}_\pair=(i',j')$ and $p'=\nextd{p}_\pair=(i',j)$.
Assume that $\prpd_\pair(p)$ is not satisfied. Then
\begin{enumerate}
\item $p'$ is $\downarrow$-maximal with respect to $\pair$.
\item For any  $p''=(i_1,j)$ with $i\le i_1\le i'$ we have $\rk_\pair(p'')\ge\rk_\pair(p)$ and if equality holds then
$\nextd{p''}_\pair=p'$, $\rshft{p''}_\pair=(i_1,j')$ and $\prpd_\pair(p'')$ is not satisfied.
\item In particular, $\rshft{p'}_\pair=\nextrd{p}_\pair$.
\item If $\qrpd_\pair(p)$ is satisfied then $\nextu{p'}_\pair=(i'',j)$ with $i''<i$.
\end{enumerate}
\end{lemma}

The following lemma is straightforward.
\begin{lemma} \label{lem: simpqp'''}
Let $p=(i,j)\in\Xc_\pair$ and $t=t_{i_1,i_2}\in\RR_\pair$. Assume that $q=\rshft p_{\pair_t}\in\Xc_\pair$
(i.e., $q\in\Xc_{(xt,x)}$).
Then $\lshft q_\pair=\lshft p_\pair$. Assume that $\lshft\prp_\pair(q)$ holds.
Then $\nextu{p}_\pair=(i',j)$ with $i'<i_1$ and in fact $i'\le i''$ where $\lushft q_\pair=(i'',j'')$.
In particular, $p$ is not $\uparrow$-maximal with respect to $\pair$.
If in addition $p\in\Xc_{(xt,x)}$ then $\pnt_x(i_1)\in (\lushft q_\pair,p]$ and hence by Lemma \ref{lem: prpdp} $\prpu_\pair(p)$ holds.
\end{lemma}

\begin{lemma} \label{lem: maxsmlr'}
Let $t\in\RR_\pair$, $p\in\Xc_{\pair_t}$, $\nextd{p}_\pair=(i',j)$ and $\nextd{p}_{\pair_t}=(i'',j)$.
Then $i''\ge i'$. In particular, if $p$ is $\downarrow$-maximal with respect to $\pair_t$
then it is also $\downarrow$-maximal with respect to $\pair$.
\end{lemma}

\begin{proof}
Write $q=\nextrd{p}_\pair=(i',j')$, $p'=\rshft p_\pair=(i,j')\in\X_\pair$, $\nextrd{p}_{\pair_t}=(i'',j'')$,
$p''=\rshft p_{\pair_t}=(i,j'')\in\X_{\pair_t}$.
If $p'=p''$ the lemma is clear since $\X_\pair\subset\X_{\pair_t}$. Otherwise, $j<j''<j'$.
We have $\diff_w(q,p'')=0$ and $q,p''\in\X_{\pair_t}$. Thus, by \eqref{sf} $(i',j'')\in\X_{\pair_t}$ and hence $i''\ge i'$.
\end{proof}

\begin{lemma} \label{lem: simpqp2''}
Let $t\in\RR_\pair$, $p=(i,j)\in\Xc_\pair$, $q=\rshft p_{\pair_t}$.
Assume that $q\in\Xc_\pair$ and that $\rshft\prp_\pair(q)$ is satisfied.
If $\nextrd{p}_{\pair_t}\in\X_\pair$ then $\prpd_\pair(p)$ holds.
Otherwise, $\nextd{p}_\pair=\nextd{p}_{\pair_t}$ and the conditions $\prpd_\pair(p)$ and $\prpd_{\pair_t}(p)$ are equivalent.
\end{lemma}

\begin{proof}
Let $q=(i,j')$, $q'=\rdshft q_\pair=(i',j'')$ and $p'=\nextrd{p}_{\pair_t}=(i'',j')$. Note that $\rshft p_\pair=\rshft q_\pair$.

Suppose first that $p'\in\X_\pair$. Then $i''\ge i'$.
By the condition $\rshft\prp_\pair(q)$ we get $\diff_w(p,q')=\diff_w(p,(i',j'))+\diff_w(q,q')=0$ and $\diff_x(p,q')\ge\diff_x(q,q')=1$.
Thus, the lemma follows from Lemma \ref{lem: prpdp}.

Suppose now that $p'\in\Xc_\pair$. Then $i''<i_2\le i'$. We have $\diff_w(p,(i'',j''))=\diff_w(p,p')+\diff_w(q,(i'',j''))=0$.
Also $(i'',j'')\in\X_\pair$ since $(i,j'')\in\X_\pair$, $q,p'\in\X_{\pair_t}$ and $\diff_x(q,(i'',j''))=0$.
Hence $\nextd{p}_\pair=(i''',j)$ with $i'''\ge i''$.
On the other hand by Lemma \ref{lem: maxsmlr'} we have $i'''\le i''$.
Thus, $\nextd{p}_\pair=\nextd{p}_{\pair_t}$ the conditions $\prpd_\pair(p)$ and $\prpd_{\pair_t}(p)$ are equivalent
since $i_1\le i\le i''<i_2$.
\end{proof}

For convenience, we record the symmetric (equivalent) forms of Lemmas \ref{lem: simpqp'''}, \ref{lem: maxsmlr'} and \ref{lem: simpqp2''}.
\begin{lemma} \label{lem: simpqp''}
Let $p=(i,j)\in\Xc_\pair$ and $t=t_{i_1,i_2}\in\RR_\pair$. Assume that $q=\lshft p_{\pair_t}\in\Xc_\pair$.
Then $\rshft q_\pair=\rshft p_\pair$. Assume that $\rshft\prp_\pair(q)$ holds.
Then $\nextd{p}_\pair=(i',j)$ with $i'\ge i_2$ and in fact $i'\ge i''$ where $\rdshft q_\pair=(i'',j'')$.
In particular, $p$ is not $\downarrow$-maximal with respect to $\pair$.
If in addition $p\in\Xc_{(xt,x)}$ then $\pnt_x(i_2)\in (p,\rdshft q_\pair]$ and hence
by Lemma \ref{lem: prpdp} $\prpd_\pair(p)$ holds.
\end{lemma}

\begin{lemma} \label{lem: maxsmlr}
Let $t\in\RR_\pair$, $p\in\Xc_{\pair_t}$, $\nextu{p}_\pair=(i',j)$ and $\nextu{p}_{\pair_t}=(i'',j)$.
Then $i''\le i'$. In particular, if $p$ is $\uparrow$-maximal with respect to $\pair_t$
then it is also $\uparrow$-maximal with respect to $\pair$.
\end{lemma}

\begin{lemma} \label{lem: simpqp2'}
Let $t\in\RR_\pair$, $p=(i,j)\in\Xc_\pair$, $q=\lshft p_{\pair_t}$.
Assume that $q\in\Xc_\pair$ and that $\lshft\prp_\pair(q)$ is satisfied.
If $\nextlu{p}_{\pair_t}\in\X_\pair$ then $\prpu_\pair(p)$ holds.
Otherwise, $\nextu{p}_\pair=\nextu{p}_{\pair_t}$ and the conditions $\prpu_\pair(p)$ and $\prpu_{\pair_t}(p)$ are equivalent.
\end{lemma}

We need a couple of more Lemmas.

\begin{lemma} \label{lem: new1}
Let $p\in\Xc_\pair$, $q=\nextu{p}_\pair$ and $t\in\RR_\pair$. Suppose that $q\in\Xc_{(xt,x)}$ but $p\in\X_{(xt,x)}$.
Assume that $\prpu_\pair(p)$ is not satisfied but $\qrpd_\pair(q)$ is satisfied. Then $\prpd_\pair(q)$ is satisfied.
\end{lemma}

\begin{proof}
Let $p=(i,j)$, $q=(i',j)$ and $q'=\nextd{q}_\pair=(i'',j)$. Write $t=t_{i_1,i_2}$ and $t^x=t_{j_1,j_2}$.
By assumption $i\ge i_2>i'$ and $j_1\le j<j_2$.
Assume on the contrary that $\prpd_\pair(q)$ is not satisfied. Then $i''<i_2\le i$, for otherwise
$\pnt_x(i_2)\in (q,\nextrd{q}_\pair]$ in contradiction with \eqref{prpdholds}. Since $\rk_\pair(q)=\rk_\pair(q')$, the condition
$\prpu_\pair(q')$, guaranteed by $\qrpd_\pair(q)$, is inconsistent with Lemma \ref{lem: nitm2}, part \ref{part: two}.
\end{proof}

\begin{lemma} \label{lem: mxs}
Let $t=t_{i_1,i_2}\in\RR_\pair$, $p=(i,j)\in\Xc_{\pair_t}$ and $q=\nextu{p}_{\pair_t}=(i',j)$.
Assume that $q\in\Xc_{(xt,x)}$ and $\lshft p_\pair=(i,j')$ with $j'\ge j_1=x(i_1)$.
Then $q$ is $\uparrow$-maximal with respect to $\pair$.
\end{lemma}

\begin{proof}
Assume on the contrary that $q$ is not $\uparrow$-maximal with respect to $\pair$ and let $q''=\nextlu{q}_\pair=(i'',j'')$,
$i''<i'$ and $q'=\nextlu{p}_{\pair_t}=(i',j')$.
Then $j''<j_1\le j'$ and $\diff_w(q'',q')=0$. Since $q'',q'\in\X_{\pair_t}$ we must have $(i'',j')\in\X_{\pair_t}$ by \eqref{sf}.
This violates the definition of $i'$ since $\diff_w((i'',j'),q)=0$.
\end{proof}

\subsection{Induction step}
We will prove Proposition \ref{prop: main} by induction on $\ell(\pair)$.
Unfortunately, the induction step splits into many cases and subcases.
We analyze the different cases in Lemmas \ref{lem: specase12}-\ref{lem: caseD} below.

For the rest of this (long) subsection we fix a smooth pair $\pair=(w,x)\in\pairset$ and $t\in\RRs_\pair$ and assume that
\begin{equation} \label{AI} \tag{IH}
\text{$\qrp_{\pair_t}(p)$ is satisfied for any $p\in\Xc_{\pair_t}$.}
\end{equation}
We will write $t=t_{i_1,i_2}$ and $t^x=t_{j_1,j_2}$.

\begin{subtheorem}{theorem}
\begin{lemma} \label{lem: specase12}
$\qrp_\pair(p)$ holds for any $p=(i,j)\in\Xc_{(xt,x)}$.
\end{lemma}

\begin{proof}
Suppose first that $\rk_\pair(p)=1$. Then $p$ is in the critical set $\crit_\pair$ \eqref{def: critset}.
In view of Lemma \ref{lem: prpdp}, this case follows from Proposition \ref{prop: main1'} part \ref{part: mainr}.

For the rest of the proof we assume that $\rk_\pair(p)>1$.

Note that $\rk_{\pair_t}(p)=\rk_\pair(p)-1>0$.
By \eqref{AI} $\qrp_{\pair_t}(p)$ holds.
By passing to $\pair^*$, $p^*$ and $t^*$ if necessary we may assume without loss of generality that $\qrpd_{\pair_t}(p)$ is satisfied.

Write $\nextrd{p}_{\pair_t}=(i',j')$ and $p'=\nextd{p}_{\pair_t}=(i',j)$.

Suppose first that $i'\ge i_2$.

\begin{enumerate}
\item If also $j'\ge j_2$ then $\nextrd{p}_\pair=\nextrd{p}_{\pair_t}$
and $\prpd_\pair(p)$ is satisfied (even if $\prpd_{\pair_t}(p)$ is not satisfied)
since $\pnt_x(i_2)\in (p,\nextrd{p}_\pair]$.

\item Suppose that $j'<j_2$.

We have $\nextrd{p}_{\pair_t}\in\X_\pair$.
Let $q=\rshft{p}_{\pair_t}=(i,j')\in\X_{\pair_t}\setminus\X_\pair$.
Then $\lshft\prp_\pair(q)$ implies $\prpu_\pair(p)$ by Lemma \ref{lem: simpqp'''} while $\rshft\prp_\pair(q)$ implies $\prpd_\pair(p)$
by Lemma \ref{lem: simpqp2''}.
We conclude by Proposition \ref{prop: main1'} part \ref{part: mainr}.

\end{enumerate}

For the rest of the proof we assume that $i'<i_2$ so that $p'=(i',j)\in\Xc_{(xt,x)}$.
We also assume, as we may, that $\prpd_\pair(p)$ is not satisfied.

If $j'<j_2$ then we can assume by Proposition \ref{prop: main1'} part \ref{part: mainr}, applied to $(i,j')\in\X_{\pair_t}\setminus\X_\pair$,
that $\rshft\prp_\pair(i,j')$ is satisfied
since by Lemma \ref{lem: simpqp'''} $\lshft\prp_\pair(i,j')$ implies $\prpu_\pair(p)$.
By Lemma \ref{lem: simpqp2''} and our assumption,
\[
\text{$p'=\nextd{p}_\pair$ and $\prpd_{\pair_t}(p)$ is not satisfied.}
\]
The last line is also valid if $j'\ge j_2$, in which case $\nextrd{p}_\pair=\nextrd{p}_{\pair_t}$.

Thus, $\prpu_{\pair_t}(p')$ holds and $p'$ is $\downarrow$-maximal with respect to $\pair$.

Let $\nextlu{p'}_{\pair_t}=(i'',j'')$ and $p''=\lshft{p'}_{\pair_t}=(i',j'')$.
We separate the argument according to cases.
\begin{enumerate}
\item Suppose that $j''<j_1$. Then $\nextlu{p'}_\pair=\nextlu{p'}_{\pair_t}$ and hence
$\prpu_\pair(p')$ follows from $\prpu_{\pair_t}(p')$, so that $\qrpd_\pair(p)$ holds.
\item Suppose that $j''\ge j_1$ and $i''<i_1$.

By Lemma \ref{lem: simpqp''} $\rshft\prp_\pair(p'')$ doesn't hold.
Thus by Proposition \ref{prop: main1'} part \ref{part: mainr}, $\lshft\prp_\pair(p'')$ is satisfied
which by Lemma \ref{lem: simpqp2'} implies $\prpu_\pair(p')$, hence $\qrpd_\pair(p)$.

\item Finally we cannot have both $i''\ge i_1$ and $j''\ge j_1$ since otherwise, as
$\prpu_{\pair_t}(p')$ holds, there would exist $i_3$ such that $\pnt_x(i_3)\in (\nextlu{p'}_{\pair_t},p']\subset (\pnt_x(i_1),\pnt_x(i_2))$
violating the assumption that $\tau\in\RRs_\pair$.
\end{enumerate}
The proof of the lemma is complete.
\end{proof}

\begin{lemma} \label{lem: caseA}
$\qrp_\pair(p)$ holds for any $p=(i,j)\in\Xc_\pair$ such that that $i_2\le i$ and $j_1\le j<j_2$.
\end{lemma}

\begin{proof}
Note that $\rk_{\pair_t}(p)=\rk_\pair(p)$ and hence by \eqref{AI}, $\qrp_{\pair_t}(p)$ is satisfied.
Under our condition on $t$ we have $\nextrd{p}_{\pair_t}=\nextrd{p}_\pair$,
and hence the conditions $\prpd_{\pair_t}(p)$ and $\prpd_\pair(p)$ are equivalent.
Therefore, we may assume that they are not satisfied.

Suppose first that $\prpu_{\pair_t}(p')$ holds where $p'=p$ or $p'=\nextd{p}_\pair=\nextd{p}_{\pair_t}$.
Write $p'=(i',j)$ and recall that $\rk_\pair(p')=\rk_\pair(p)$.
Let $q=\nextlu{p'}_{\pair_t}=(i'',j'')\in\X_{\pair_t}$ and $p''=\nextu{p'}_{\pair_t}=(i'',j)$
so that $\lshft{p'}_\pair=(i',j'')$ and $i''<i\le i'$.
Note that $\lshft{p'}_{\pair_t}=\lshft{p'}_\pair$ by the condition on $t$.
If $p''\in\X_{(xt,x)}$ (which means that either $i''<i_1$ or $i''\ge i_2$)
then $q\in\X_\pair$, $\nextu{p'}_\pair=\nextu{p'}_{\pair_t}$,
$\diff_x(q,p')=\diff_{xt}(q,p')$ and hence $\prpu_\pair(p')$ follows from $\prpu_{\pair_t}(p')$.

Thus, we may assume that $i_1\le i''<i_2$, i.e. that $p''\in\Xc_{(xt,x)}$.
\begin{enumerate}
\item Assume first that $q\in\X_\pair$ (i.e., that $q\in\X_{(xt,x)}$, or equivalently that $j_1>j''$).

Then $\nextu{p'}_\pair=p''$.
We may assume that $\prpu_\pair(p')$ is not satisfied, for otherwise $\qrpd_\pair(p)$ holds.
By Lemma \ref{lem: nitm2} $p''$ is $\uparrow$-maximal with respect to $\pair$,
$\nextu{p}_\pair=\nextu{p'}_\pair=p''$ and $\prpu_\pair(p)$ is not satisfied.
It follows from Lemma \ref{lem: specase12} and \eqref{impluord} that $\qrpd_\pair(p'')$ holds.
It follows from Lemma \ref{lem: new1} that $\prpd_\pair(p'')$ is satisfied and therefore so does $\qrpu_\pair(p)$.

\item Now assume that $q\in\Xc_\pair$, that is $j_1\le j''$.

By Lemma \ref{lem: mxs} $p''$ is $\uparrow$-maximal with respect to $\pair$.
Once again, by Lemma \ref{lem: specase12} and \eqref{impluord} $\qrpd_\pair(p'')$ holds. Let $\nextrd{p''}_\pair=(i''',j''')$ and $p'''=\nextd{p''}_\pair=(i''',j)$.
Note that $j'''\ge j_2$. Also, since $\prpu_{\pair_t}(p')$ is satisfied we have $\rk_\pair(p'')=\rk_{\pair_t}(p'')+1\le
\rk_{\pair_t}(p)=\rk_\pair(p)$.
\begin{enumerate}
\item Suppose that $\prpd_\pair(p'')$ holds.
\begin{enumerate}
\item Suppose that $i'''\le i'$.

Then $\diff_w(q,p''')=0$ and hence $\diff_\pair(q,p''')\ge0$.
On the other hand by assumption, $\rk_\pair(p''')<\rk_\pair(p'')$  and $q\in\X_{\pair_t}\setminus\X_\pair$.
Therefore, $(i''',j'')\in\X_\pair$ and $\rk_\pair(p''')=\rk_\pair(p'')-1$.
By Lemma \ref{lem: prpdp} this implies that $\prpu_\pair(p')$ is satisfied.
\item Suppose that $i'''>i'$. Let
\[
i_3=\min\{k>i'':(k,j'')\in\X_\pair\}.
\]
Note that $i_2\le i_3\le i'<i'''$ and that $\rshft{q}_\pair=\rshft{p''}_\pair$. Let $q'=(i_3,j''')$.
Since $\diff_w(q,q')=\diff_w(q,(i_3,j))+\diff_w(p'',q')=0$
we may apply Proposition \ref{prop: main1'} part \ref{part: suppl} to $q$ to infer that $q'\in\X_\pair$
and $\diff_x(q,q')=1$, hence $\diff_x(p'',q')=1$ as $\pnt_x(i_2)\in (p'',q']$.
Hence, $\rk_\pair(i_3,j)=\rk_\pair(p'')-1<\rk_\pair(p')$.
As before, it follows from Lemma \ref{lem: prpdp} that $\prpu_\pair(p')$ is satisfied.
\end{enumerate}
\item Suppose that $\prpd_\pair(p'')$ is not satisfied.

Then $\rk_\pair(p''')=\rk_\pair(p'')$, $i'''<i_2$ (for otherwise $\pnt_x(i_2)\in (p'',\nextrd{p''}_\pair]$)
and $\rk_\pair(i_4,j)\ge\rk_\pair(p'')$ for any $i''\le i_4\le i'''$.
Let $q'':=\nextlu{p'''}_\pair=(i_5,j_5)\in\X_\pair$ so that $\nextu{p'''}_\pair=(i_5,j)$.
Then $j_5<j_1\le j''$. Since $\qrpd_\pair(p'')$ holds, $\prpu_\pair(p''')$ is satisfied and therefore $\rk_\pair(\nextu{p'''}_\pair)<\rk_\pair(p'')$.
Hence $i_5<i''$ and therefore $q''<q$.
Since $\diff_w(q'',q)=0$ and $q\in\X_{\pair_t}$ we necessarily have $(i_5,j'')\in\X_{\pair_t}$.
Since $\diff_w(q'',p'')=0$ we would get a contradiction to the fact that $p''=\nextu{p'}_{\pair_t}$.
\end{enumerate}
\end{enumerate}

It remains to consider the case where $\prpu_{\pair_t}(p)$ does not hold but $\prpd_{\pair_t}(p')$ holds for $p'=\nextu{p}_{\pair_t}$.

Write $q=\nextlu{p}_{\pair_t}=(i',j')\in\X_{\pair_t}$ and $p'=\nextu{p}_{\pair_t}=(i',j)$ and note that $\lshft{p}_\pair=\lshft{p}_{\pair_t}=(i,j')$.

As before, if $p'\in\X_{(xt,x)}$, i.e., if either $i'<i_1$ or $i'\ge i_2$ then
$p'=\nextu{p}_\pair$, $q=\nextlu{p}_\pair$ and $\diff_x(q,p)=\diff_{xt}(q,p)$.
Moreover, since $\nextd{p'}_{\pair_t}=(i'',j)$ with $i''>i$ we have $\nextrd{p'}_{\pair_t}=\nextrd{p'}_\pair$ and
$\diff_x(p',\nextrd{p'}_\pair)=\diff_{xt}(p',\nextrd{p'}_\pair)$.
Hence $\qrpu_\pair(p)$ follows from $\qrpu_{\pair_t}(p)$.

Assume therefore that $i_1\le i'<i_2$, i.e., $p'\in\Xc_{(xt,x)}$.
Since $\prpu_{\pair_t}(p)$ does not hold (by assumption), we necessarily have $j'\ge j_1$
for otherwise $\pnt_{xt}(i_2)\in (\nextlu{p}_{\pair_t},p]$.

It follows from the definitions that $p'':=\nextu{p}_\pair=(i_3,j)$ where
\[
i_3=\min\{k>i':(k,j')\in\X_\pair\}.
\]
(with $i_2\le i_3\le i$).
By Lemma \ref{lem: nitm2} we have $\rk_\pair(p'')=\rk_{\pair_t}(p'')=\rk_{\pair_t}(p')=\rk_\pair(p')-1=\rk_\pair(p)=\rk_{\pair_t}(p)$.

Write $\nextrd{p'}_{\pair_t}=(i'',j'')$ so that $\rshft{p'}_{\pair_t}=(i',j'')\in\X_{\pair_t}$ and recall that $i''>i$.

\begin{enumerate}
\item Suppose that $j''\ge j_2$.

Then $\rshft{q}_\pair=\rshft{p'}_\pair=\rshft{p'}_{\pair_t}$.
By Proposition \ref{prop: main1'} part \ref{part: suppl} (applied to $q$) we have $(i_3,j'')\in\X_\pair$
and $\diff_{xt}(q,(i_3,j''))=0$.
Thus, $\rshft{p''}_\pair=(i_3,j_3)$ with $j_3\le j''$. We claim that in fact $j_3=j''$.
Indeed, if $j_3<j''$ then necessarily $(i',j_3)\in\X_{\pair_t}$ since $\diff_{xt}((i',j_3),(i_3,j''))=\diff_w((i',j_3),(i_3,j''))=0$
and $(i_3,j_3),(i_3,j''),(i',j'')\in\X_{\pair_t}$, and this contradicts the definition of $j''$.

We conclude from the condition $\prpd_{\pair_t}(p')$ that $\prpd_\pair(p'')$ holds.
Thus $\qrpu_\pair(p)$ or $\prpd_\pair(p)$ is satisfied depending on whether $i_3<i$ or $i_3=i$.

\item Suppose that $j''<j_2$.

Since $p'$ is $\uparrow$-maximal with respect to $\pair_t$ (Lemma \ref{lem: nitm2}), it is also $\uparrow$-maximal with respect to $\pair$
(Lemma \ref{lem: maxsmlr}).
Let $q'=\rshft{p'}_{\pair_t}\in\X_{\pair_t}\setminus\X_\pair$.
By Lemma \ref{lem: simpqp'''} we cannot have $\lshft\prp_\pair(q')$.
Hence, by Proposition \ref{prop: main1'} part \ref{part: mainr} we have $\rshft\prp_\pair(q')$.
By Lemma \ref{lem: simpqp2''} we conclude from $\prpd_{\pair_t}(p')$ that $\prpd_\pair(p')$ holds.
\end{enumerate}
This finishes the proof of the lemma.
\end{proof}

\begin{lemma} \label{lem: caseB}
$\qrp_\pair(p)$ holds for any $p=(i,j)\in\Xc_\pair$ such that $i<i_1<i_2\le i_0$ and $j_1\le j<j_2$ where
$\nextd{p}_\pair=(i_0,j)$.
\end{lemma}

\begin{proof}
Note that $\rk_{\pair_t}(p)=\rk_\pair(p)$ and in particular $p\in\Xc_{\pair_t}$.
Hence by \eqref{AI}, $\qrp_{\pair_t}(p)$ is satisfied.
Under our condition on $t$ we have 
$\nextrd{p}_{\pair_t}=\nextrd{p}_\pair$, $\diff_x(p,\nextrd{p}_\pair)=\diff_{xt}(p,\nextrd{p}_\pair)$
and the conditions $\prpd_{\pair_t}(p)$ and $\prpd_\pair(p)$ are equivalent.
Therefore, we may assume that $\prpd_{\pair_t}(p)$ (and $\prpd_\pair(p)$) are not satisfied.

Suppose first that $\prpu_{\pair_t}(p')$ holds where $p'=p$ or $p'=\nextd{p}_\pair=\nextd{p}_{\pair_t}$.
Write $p'=(i',j)$ and note that $\rk_\pair(p')=\rk_\pair(p)$.
Let $q=\nextlu{p'}_{\pair_t}=(i'',j'')\in\X_{\pair_t}$ and $p''=\nextu{p'}_{\pair_t}=(i'',j)$.
By the condition on $t$ we have $\lshft{p'}_{\pair_t}=\lshft{p'}_\pair$.
On the other hand, by Lemma \ref{lem: nitm} (applied to $\pair_t$) we have $i''<i\le i'$.
Hence $p''\in\X_{(xt,x)}$, $q\in\X_\pair$, $\nextlu{p'}_{\pair_t}=\nextlu{p}_\pair$ and the condition $\prpu_\pair(p')$ follows from $\prpu_\pair(p)$.

It remains to consider the case where $\prpu_{\pair_t}(p)$ does not hold but $\prpd_{\pair_t}(p')$ holds for $p'=\nextu{p}_{\pair_t}=(i'',j)$.
In particular, $\rk_\pair(p')=\rk_\pair(p)$.
(Of course, $\rk_{\pair_t}(p')=\rk_\pair(p')$ and $\rk_{\pair_t}(p)=\rk_\pair(p)$ since $i<i_1$.)

Write $q'=\rshft{p}_\pair=(i,j_0)\in\X_\pair$ and $q=\nextrd{p'}_{\pair_t}=(i',j')$.
Since $i<i_1$ we have $p'=\nextu{p}_\pair$ and $\rshft{p'}_\pair=\rshft{p'}_{\pair_t}$.
Also, $j_0<j_2$ for otherwise $\pnt_x(i_2)\in (p,\nextrd{p}_\pair]$ since $i_2\le i_0$, and $\prpd_\pair(p)$ would be satisfied.
By Lemma \ref{lem: nitm} (applied to $\pair_t$) we have $i'>i$.
Also, $j_0\ge j'$ for otherwise $\diff_\pair(p',q')=\diff_x(p',q')\ge0$ and $(i'',j_0)\in\Xc_\pair$
and this contradicts the fact that $q'\in\X_\pair$ and $\rk_\pair(p')=\rk_\pair(p)$.
If $q\in\X_\pair$ then $\prpd_\pair(p')$ is satisfied by Lemma \ref{lem: prpdp}, and hence $\qrpd_\pair(p)$ holds.
Otherwise, $q\in\X_{\pair_t}\setminus\X_\pair$ and $(i',j_0)\in\Xc_\pair$ (since $j_0<j_2$).
Since $\diff_x(q,\nextrd{p}_\pair)=0$ and $\nextrd{p}_\pair\in\X_\pair$, we must have $(i_0,j')\in\X_\pair$.
Hence, $\prpd_\pair(p')$ is satisfied by Lemma \ref{lem: prpdp} (and in fact $j'=j_0$).
\end{proof}

\begin{lemma} \label{lem: caseC}
$\qrp_\pair(p)$ holds for any $p=(i,j)\in\Xc_\pair$ such that $i<i_1$ and $\nextd{p}_\pair\in\Xc_{(xt,x)}$.
\end{lemma}

\begin{proof}
Write $\nextrd{p}_\pair=(i_0,j_0)\in\X_\pair$, $p'=\nextd{p}_\pair=(i_0,j)$ and note that $j_0\ge j_2$.
Also note that the condition $\prpd_{\pair_t}(p)$ is satisfied.

We may assume of course that $\prpd_\pair(p)$ is not satisfied, so that $\rk_\pair(p')=\rk_\pair(p)$.

We apply Lemma \ref{lem: specase12} to $p'$. Recall that $p'$ is $\downarrow$-maximal with respect to $\pair$.
Thus, by \eqref{impluord}  $\qrpu_\pair(p')$ is satisfied.
Since $\prpu_\pair(p')$ implies $\qrpd_\pair(p)$, we may assume that $\prpu_\pair(p')$ is not satisfied.
Let $p''=\nextu{p'}_\pair=(i',j)$.
Then $\rk_\pair(p'')=\rk_\pair(p')=\rk_\pair(p)$ and $\prpd_\pair(p'')$ holds.
Let $\nextlu{p'}_\pair=(i',j')$ so that $\lshft{p'}_\pair=(i_0,j')$. Then $j'<j_1$.
Also $i'\ge i_1>i$ for otherwise $\pnt_x(i_1)\in (\nextlu{p'}_\pair,p']$ which would imply $\prpu_\pair(p')$.
By Lemma \ref{lem: nitm} this would contradict $\prpd_\pair(p'')$.
\end{proof}

\begin{lemma} \label{lem: caseD}
$\qrp_\pair(p)$ holds for any $p=(i,j)\in\Xc_\pair$ such that $i_0<i_2$ and $j_2\le j$ where $\nextd{p}_\pair=(i_0,j)$.
\end{lemma}

\begin{proof}
We have $\rk_{\pair_t}(p)=\rk_\pair(p)$ and in particular $p\in\Xc_{\pair_t}$. By \eqref{AI}, $\qrp_{\pair_t}(p)$ is satisfied.
Note that the conditions $\prpd_\pair(p)$ and $\prpd_{\pair_t}(p)$ are equivalent. We may therefore assume that they do not hold.
We separate according to cases.
\begin{enumerate}
\item Suppose that $\prpu_{\pair_t}(p')$ holds for $p'=\nextd{p}_\pair=\nextd{p}_{\pair_t}=(i_0,j)$.

Let $q=\nextlu{p'}_{\pair_t}=(i',j')$ and $p''=\lshft{p'}_{\pair_t}=(i_0,j')\in\X_{\pair_t}$.
\begin{enumerate}
\item Suppose that $p''\in\X_\pair$.

Then $p''=\lshft{p'}_\pair$ and since $i_2>i_0$ we have $\nextlu{p'}_\pair=q$, $\nextu{p'}_\pair\in\X_{(xt,x)}$.
The condition $\prpu_\pair(p')$ (and hence, $\qrpd_\pair(p)$) follows from $\prpu_{\pair_t}(p')$.

\item Assume that $p''\in\Xc_\pair$.

By Lemma \ref{lem: simpqp''} the condition $\rshft\prp_\pair(p'')$ would contradict
the fact that $p'$ is $\downarrow$-maximal with respect to $\pair$.

Hence by Proposition \ref{prop: main1'} part \ref{part: mainr} $\lshft\prp_\pair(p'')$ holds.
By Lemma \ref{lem: simpqp2'} we have $\prpu_\pair(p')$ and hence $\qrpd_\pair(p)$.

\end{enumerate}
\item Suppose that $\prpu_{\pair_t}(p)$ holds.

Let $q=\nextlu{p}_{\pair_t}=(i',j')$ and $p'=\lshft{p}_{\pair_t}=(i,j')$.

\begin{enumerate}
\item Suppose that $p'\in\X_\pair$.

Then $\lshft{p}_\pair=p'$, $\nextlu{p}_\pair=q$, $\nextu{p}_\pair\in\X_{(xt,x)}$.
Hence $\prpu_\pair(p)$ follows from $\prpu_{\pair_t}(p)$.

\item Suppose that $p'\in\Xc_\pair$.

By Lemma \ref{lem: simpqp''} we cannot have $\rshft\prp_\pair(p')$ since $i_2>i_0$.

Hence, by Proposition \ref{prop: main1'} part \ref{part: mainr} $\lshft\prp_\pair(p')$ holds.
By Lemma \ref{lem: simpqp2'} we have $\prpu_\pair(p)$.
\end{enumerate}

\item Suppose that $\prpu_{\pair_t}(p)$ does not hold but $\prpd_{\pair_t}(\nextu{p}_{\pair_t})$ holds.

Again, let $q=\nextlu{p}_{\pair_t}=(i',j')$ and $p'=\lshft{p}_{\pair_t}=(i,j')$.

\begin{enumerate}
\item As before, if $p'\in\X_\pair$ then $\lshft{p}_\pair=p'$,
$\nextlu{p}_\pair=q$, $\nextu{p}_\pair=\nextu{p}_{\pair_t}$
and $\qrpu_\pair(p)$ follows from $\qrpu_{\pair_t}(p)$.

\item Suppose that $p'\in\Xc_\pair$.
By Lemma \ref{lem: simpqp''} we cannot have $\rshft\prp_\pair(p')$ since $i_2>i_0$.

Hence, by Proposition \ref{prop: main1'} part \ref{part: mainr} $\lshft\prp_\pair(p')$ holds.
By Lemma \ref{lem: simpqp2'} if $q\in\X_\pair$ then $\prpu_\pair(p)$ holds,
or otherwise, $\nextu{p}_\pair=\nextu{p}_{\pair_t}$, and hence $\qrpu_\pair(p)$ is satisfied.
\end{enumerate}
\end{enumerate}
\end{proof}

\end{subtheorem}

\subsection{Conclusion of proof}
In order to conclude the proof of \ref{prop: main} we need to show that the cases covered in
the previous subsection are exhaustive.

Let
\[
\auxs_\pair(p)=\{t_{i_1,i_2}\in\RR_\pair:\pnt_x(i_1)\le p\text{ but }\pnt_x(i_2)\not\le p\}
\]
and $\auxss_\pair(p)=\auxs_\pair(p)\cap\RRs_\pair$.

\begin{lemma} \label{lem: simptau}
$\auxss_\pair(p)\ne\emptyset$ for any $p\in\Xc_\pair$.
\end{lemma}

\begin{proof}
We first remark that it suffices to show that $\auxs_\pair(p)\ne\emptyset$ for any $p\in\Xc_\pair$
since if $t_{i_1,i_2}\in\auxs_\pair(p)$ then for any $i'$ such that $t_{i_1,i'},t_{i',i_2}\in\RR_\pair$ we have either $t_{i,i'}\in\auxs_\pair(p)$
or $t_{i',j}\in\auxs_\pair(p)$.
We prove the statement by induction on $\ell(\pair)$.
The statement is empty if $\ell(\pair)=0$.
For the induction step, assume that $\ell(\pair)>0$ and observe that if $t\in\RR_\pair$ and $p\in\Xc_{(xt,x)}$ then $t\in\auxs_\pair(p)$.
Let $t'\in\RRs_\pair$ and $p\in\Xc_\pair$. If $t'\in\auxs_\pair(p)$ we are done.
Otherwise, $p\in\Xc_{\pair_{t'}}$ and by induction hypothesis $\auxs_{\pair_{t'}}(p)\ne\emptyset$.
However, it is straightforward to check that the map $\phi_{t'}^\pair$ of \S\ref{sec: gash} satisfies
$\phi_{t'}^\pair(\auxs_{\pair_{t'}}(p))\subset\auxs_\pair(p)$. 
\end{proof}

\begin{proof}[Proof of Proposition \ref{prop: main}]
We will prove the proposition by induction on $\ell(\pair)=\ell(x)-\ell(w)\ge0$.
The case $\ell(w)=\ell(x)$ (i.e., $w=x$) is of course trivial.
Let $p\in\Xc_\pair$. We may assume of course that $\prpd_\pair(p)$ is not satisfied. Write $p'=\nextd{p}_\pair=(i_0,j)$.
Since $\rk_\pair(p')=\rk_\pair(p)$ we have $p'\in\Xc_\pair$. By Lemma \ref{lem: simptau} $\auxss_\pair(p')\ne\emptyset$.
Let $t\in\auxss_\pair(p')$. The induction hypothesis implies \eqref{AI}.
Writing $t=t_{i_1,i_2}$ and $t^x=t_{j_1,j_2}$, exactly one of the following conditions hold.
\begin{enumerate}
\item $p\in\Xc_{(xt,x)}(=[\pnt_x(i_1),\pnt_x(i_2)))$.
\item $i_2\le i$ and $j_1\le j<j_2$.
\item $i<i_1<i_2\le i_0$ and $j_1\le j<j_2$.
\item $i<i_1$ and $p'\in\Xc_{(xt,x)}$
\item $i_1\le i_0<i_2$ and $j_2\le j$.
\end{enumerate}
We apply Lemmas \ref{lem: specase12}, \ref{lem: caseA}, \ref{lem: caseB}, \ref{lem: caseC} and \ref{lem: caseD}
respectively to conclude in each case that $\qrp_\pair(p)$ holds as required.
\end{proof}


\def\cprime{$'$} \def\Dbar{\leavevmode\lower.6ex\hbox to 0pt{\hskip-.23ex
  \accent"16\hss}D} \def\cftil#1{\ifmmode\setbox7\hbox{$\accent"5E#1$}\else
  \setbox7\hbox{\accent"5E#1}\penalty 10000\relax\fi\raise 1\ht7
  \hbox{\lower1.15ex\hbox to 1\wd7{\hss\accent"7E\hss}}\penalty 10000
  \hskip-1\wd7\penalty 10000\box7}
  \def\polhk#1{\setbox0=\hbox{#1}{\ooalign{\hidewidth
  \lower1.5ex\hbox{`}\hidewidth\crcr\unhbox0}}} \def\dbar{\leavevmode\hbox to
  0pt{\hskip.2ex \accent"16\hss}d}
  \def\cfac#1{\ifmmode\setbox7\hbox{$\accent"5E#1$}\else
  \setbox7\hbox{\accent"5E#1}\penalty 10000\relax\fi\raise 1\ht7
  \hbox{\lower1.15ex\hbox to 1\wd7{\hss\accent"13\hss}}\penalty 10000
  \hskip-1\wd7\penalty 10000\box7}
  \def\ocirc#1{\ifmmode\setbox0=\hbox{$#1$}\dimen0=\ht0 \advance\dimen0
  by1pt\rlap{\hbox to\wd0{\hss\raise\dimen0
  \hbox{\hskip.2em$\scriptscriptstyle\circ$}\hss}}#1\else {\accent"17 #1}\fi}
  \def\bud{$''$} \def\cfudot#1{\ifmmode\setbox7\hbox{$\accent"5E#1$}\else
  \setbox7\hbox{\accent"5E#1}\penalty 10000\relax\fi\raise 1\ht7
  \hbox{\raise.1ex\hbox to 1\wd7{\hss.\hss}}\penalty 10000 \hskip-1\wd7\penalty
  10000\box7} \def\lfhook#1{\setbox0=\hbox{#1}{\ooalign{\hidewidth
  \lower1.5ex\hbox{'}\hidewidth\crcr\unhbox0}}}
\providecommand{\bysame}{\leavevmode\hbox to3em{\hrulefill}\thinspace}
\providecommand{\MR}{\relax\ifhmode\unskip\space\fi MR }
\providecommand{\MRhref}[2]{%
  \href{http://www.ams.org/mathscinet-getitem?mr=#1}{#2}
}
\providecommand{\href}[2]{#2}

\end{document}